\documentclass[a4paper,12pt,reqno,oneside]{article}
\usepackage[utf8x]{inputenc}
\usepackage[T1]{fontenc}
\usepackage{geometry}
\usepackage{lmodern}

\usepackage{amsmath}
\usepackage{amssymb}
\usepackage{amsthm}
\usepackage[matrix,arrow,cmtip]{xy}
\usepackage{bbm}
\usepackage{enumerate}
\PassOptionsToPackage{hyphens}{url}
\usepackage[bookmarksnumbered,colorlinks,linkcolor=black,citecolor=black,urlcolor=black]{hyperref}
\usepackage[capitalise]{cleveref}
\usepackage[nottoc]{tocbibind}

\newcommand{\bA}{\mathbb{A}}
\newcommand{\bC}{\mathbb{C}}
\newcommand{\bF}{\mathbb{F}}
\newcommand{\bG}{\mathbb{G}}

\newcommand{\bP}{\mathbb{P}}
\newcommand{\bQ}{\mathbb{Q}}
\newcommand{\bR}{\mathbb{R}}
\newcommand{\bS}{\mathbb{S}}
\newcommand{\bZ}{\mathbb{Z}}

\newcommand{\kalg}{k^{\mathrm{alg}}}
\newcommand{\Qalg}{\bQ^{\mathrm{alg}}}
\newcommand{\QalgR}{\Qalg \cap \bR}
\newcommand{\QalgRb}{(\QalgR)}

\newcommand{\cA}{\mathcal{A}}
\newcommand{\cB}{\mathcal{B}}
\newcommand{\cF}{\mathcal{F}}
\newcommand{\cH}{\mathcal{H}}

\newcommand{\cO}{\mathcal{O}}

\newcommand{\cX}{\mathcal{X}}

\newcommand{\Ag}{\mathcal{A}_g}

\newcommand{\Hg}{\mathcal{H}_g}

\newcommand{\fS}{\mathfrak{S}}

\newcommand{\gG}{\mathbf{G}}
\newcommand{\gGL}{\mathbf{GL}}
\newcommand{\gGSp}{\mathbf{GSp}}

\newcommand{\gM}{\mathbf{M}}

\newcommand{\gO}{\mathbf{O}}

\newcommand{\gS}{\mathbf{S}}

\newcommand{\gT}{\mathbf{T}}
\newcommand{\gU}{\mathbf{U}}

\newcommand{\gZ}{\mathbf{Z}}

\newcommand{\rM}{\mathrm{M}}

\DeclareMathOperator{\Aut}{Aut}

\DeclareMathOperator{\disc}{disc}
\DeclareMathOperator{\End}{End}
\DeclareMathOperator{\Gal}{Gal}

\DeclareMathOperator{\Hom}{Hom}

\DeclareMathOperator{\Jac}{Jac}
\DeclareMathOperator{\Lie}{Lie}

\DeclareMathOperator{\Nm}{Nm}

\DeclareMathOperator{\Res}{Res}
\DeclareMathOperator{\Sh}{Sh}
\DeclareMathOperator{\Sing}{Sing}

\DeclareMathOperator{\Stab}{Stab}
\DeclareMathOperator{\Vol}{Vol}

\newcommand{\im}{\operatorname{Im}}
\newcommand{\re}{\operatorname{Re}}

\DeclareMathOperator{\rk}{rk}

\newcommand{\abs}[1]{\left\lvert #1 \right\rvert}

\newcommand{\defterm}[1]{\textbf{#1}}

\newtheorem{lemma}{Lemma}[section]
\newtheorem{proposition}[lemma]{Proposition}
\newtheorem{theorem}[lemma]{Theorem}

\newtheorem{corollary}[lemma]{Corollary}

\newtheorem{conjecture}[lemma]{Conjecture}
\Crefname{conjecture}{Conjecture}{Conjectures} 

\Crefname{claim}{Claim}{Claims}

\newtheorem*{lemma*}{Lemma}
\newtheorem*{proposition*}{Proposition}
\newtheorem*{theorem*}{Theorem}
\newtheorem*{corollary*}{Corollary}
\newtheorem*{claim*}{Claim}
\newtheorem*{questions*}{Questions}

\theoremstyle{definition}

\newcommand{\rH}{\mathrm{H}}

\newcommand{\basis}[1]{\cB^{(#1)}}
\newcommand{\bv}[2]{v^{(#1)}_{#2}}
\newcommand{\barbv}[2]{\bar{v}^{(#1)}_{#2}}

\newcommand{\mx}[1]{
  \ifthenelse{\equal{#1}{zT2}}{A}{
  \ifthenelse{\equal{#1}{T1}}{B}{
  \ifthenelse{\equal{#1}{T3}}{B'}{
  \ifthenelse{\equal{#1}{Z2}}{P}{
  \ifthenelse{\equal{#1}{23}}{Q}{
  \ifthenelse{\equal{#1}{13}}{R}{
  \ifthenelse{\equal{#1}{13'}}{R'}{
  \ifthenelse{\equal{#1}{34V}}{S}{
  \ifthenelse{\equal{#1}{34}}{S}{
  \errmessage{Bad mx #1}
  }}}}}}}}}
}

\usepackage{amscd}

\usepackage{ifthen}
\newcounter{constant}
\newcommand{\newC}[1]{%
   \refstepcounter{constant} C_{\theconstant}%
   \ifthenelse{\equal{#1}{*}} { } {%
      \label{C:#1}%
   }%
}
\newcommand{\refC}[1]{C_{\ref*{C:#1}}}
\newcommand{\polybound}[1]{\refC{#1-multiplier} H^{\refC{#1-exponent}}}

\newcommand{\seq}[2]{#1_1, \dotsc, #1_{#2}}

\newcommand{\hgG}{\texorpdfstring{\( \gG \)}{G}}

\title{Heights of pre-special points of Shimura varieties}
\author{Christopher Daw \and Martin Orr}

\begin{document}

\maketitle

\begin{abstract}
Let \( s \) be a special point on a Shimura variety, and \( x \) a pre-image of \( s \) in a fixed fundamental set of the associated Hermitian symmetric domain.
We prove that the height of \( x \) is polynomially bounded with respect to the discriminant of the centre of the endomorphism ring of the corresponding \( \bZ \)-Hodge structure.
Our bound is the final step needed to complete a proof of the André--Oort conjecture under the conjectural lower bounds for the sizes of Galois orbits of special points, using a strategy of Pila and Zannier.
\end{abstract}

\renewcommand{\thefootnote}{}
\footnote{Christopher Daw, Institut des Hautes Études Scientifiques, Bures-sur-Yvette, France}
\footnote{Martin Orr, Dept.\ of Mathematics, Imperial College, South Kensington, London, UK}

\section{Introduction} \label{sec:introduction}

Our aim in this paper is to prove a bound for the height of a pre-special point in a fundamental set of a Hermitian symmetric domain covering a Shimura variety.
This generalises a theorem of Pila and Tsimerman (\cite{pila-tsimerman:ao-surfaces} Theorem~3.1) concerning the heights of pre-images of CM points in a fundamental set of the Siegel upper half-space.
Our motivation for considering this bound is that it completes a strategy, originating in the work~\cite{pila-zannier} of Pila and Zannier, for a new proof of the André--Oort conjecture under the conjectural lower bounds for the sizes of Galois orbits of special points, known to hold under the Generalised Riemann Hypothesis (GRH).

The following is the precise statement of our primary bound, comparing the height of a pre-special point \( x \) with the discriminant of the centre of the endomorphism ring of the \( \bZ \)-Hodge structure associated with \( x \).
Note that the association of \( \bZ \)-Hodge structures with points of \( X \) depends on the choice of a representation of the group \( \gG \) and of a lattice in this representation -- this is the only purpose of \( \rho \) and \( E_\bZ \) in the theorem.

\begin{theorem} \label{height-disc-r-bound}
Let \( (\gG, X) \) be a Shimura datum with \( \gG \) being an adjoint group.
Let \( \rho \colon \gG \to \gGL(E) \) be a faithful self-dual \( \bQ \)-representation, and fix a lattice \( E_\bZ \subset E \).

Let \( \Gamma \subset \gG(\bQ) \) be a congruence subgroup and let \( \cF \subset X \) be a fundamental set for \( \Gamma \) as in Théorème~13.1 of~\cite{borel:arithmetic-groups}, with respect to a pre-special base point \( x_0 \).

Choose a realisation of \( X \) such that the action of \( \gG(\QalgR) \) on \( X \) is semialgebraic and defined over \( \QalgR \).
Let \( \rH(x) \) denote the multiplicative Weil height of a point in \( X(\Qalg) \) with respect to this realisation.

There are constants \( \newC{height-disc-r-multiplier} \), \( \newC{height-disc-r-exponent} \) (depending on \( \gG \), \( X \), \( \cF \), \( \rho \) and the choice of a realisation for \( X \)) such that
for all pre-special points \( x \in \cF \),
\[ \rH(x) \leq \refC{height-disc-r-multiplier} \, \abs{\disc R_x}^{\refC{height-disc-r-exponent}}, \]
where \( R_x \) is the centre of the endomorphism ring of the \( \bZ \)-Hodge structure \( \rho \circ x \).
\end{theorem} 

Our motivation for proving this theorem was to apply it to the André--Oort conjecture.
In order to do this, we prove an additional bound (\cref{mainindex}) comparing the discriminant of \( R_x \) with certain invariants associated with the Mumford--Tate group of the pre-special point.
These invariants are the same invariants which appear in lower bounds for the sizes of Galois orbits of special points.

The upshot of \cref{height-disc-r-bound,mainindex} is that the absence of lower bounds for Galois orbits of special points becomes the only remaining obstacle to an unconditional proof of the André--Oort conjecture.
Specifically, the combination of \cref{height-disc-r-bound,mainindex} together with \cite{PW06}, \cite{pila-tsimerman:ax-lindemann}, \cite{ullmo:applications}, \cite{kuy:ax-lindemann} and the strategy of \cite{pila-zannier} implies the following.

\begin{theorem} \label{andre-oort-conditional}
Let \( (\gG, X) \) be a Shimura datum and \( K \) a compact open subgroup of \(\gG(\bA_f) \).

Assume that there exist positive constants \( \newC{galois-bound-multiplier} \), \( \newC{galois-bound-badprime-base} \), \( \newC{galois-bound-torus-exponent} \), \( \newC{galois-bound-disc-exponent} \) (depending only on \( \gG \), \( X \) and \( K \)) such that, for each pre-special point \( x\in X \), its image \( [x,1] \) in the Shimura variety \( \Sh_K(\gG, X) \) satisfies
\[ \abs{\Gal(\Qalg/L_x) \cdot [x,1]} \geq \refC{galois-bound-multiplier} \, \refC{galois-bound-badprime-base}^{i(\gM)} [K_{\gM}^m:K_{\gM}]^{\refC{galois-bound-torus-exponent}} \abs{\disc L_x}^{\refC{galois-bound-disc-exponent}} \]
(where we use notations (a)--(d) from \cref{main-bound}).

Then the André--Oort conjecture (\cref{andre-oort}) holds for $\Sh_K(\gG,X)$.
\end{theorem}

It is known that the Generalised Riemann Hypothesis for CM fields implies the Galois bounds assumed in \cref{andre-oort-conditional} (see~\cite{uy:galois-bound}).
Thus \cref{andre-oort-conditional} gives a new proof of the André--Oort conjecture assuming the GRH for CM fields.
The André--Oort conjecture is already known under the GRH due to the work of Klingler, Ullmo and Yafaev (see~\cite{ky:andre-oort} and~\cite{uy:andre-oort}).

However, \cref{andre-oort-conditional} is stronger than simply ``GRH implies André--Oort'' because it is conceivable that the necessary Galois bounds could be proved without using the GRH.
Tsimerman has recently proved these bounds for the case of the moduli space of principally polarised abelian varieties \( \Ag \) (see~\cite{tsimerman:galois-bounds-ag}).
This represents an advantage of \cref{andre-oort-conditional} over the proof of the André--Oort conjecture under the GRH by Klingler, Ullmo and Yafaev, as their proof depends much more heavily on the GRH.

The \( \Ag \) case of \cref{height-disc-r-bound} was proved by Pila and Tsimerman (see \cite{pila-tsimerman:ao-surfaces} Theorem~3.1).
Earlier known cases of \cref{height-disc-r-bound} included modular curves (see~\cite{P09a}) and Hilbert modular surfaces (see~\cite{DY11}).
The analogous bound for abelian varieties in place of Shimura varieties, used in Pila and Zanner's proof of the Manin--Mumford conjecture, is trivial.

It is not difficult to deduce \cref{height-disc-r-bound} for all Shimura varieties of abelian type from the case of~\( \Ag \),
so the new contribution of this paper is that it holds for Shimura varieties which are not of abelian type.
Our proof works uniformly for all Shimura varieties, but as we explain later in the introduction, moving beyond Shimura varieties of abelian type introduced substantial new difficulties.

\subsection{The Pila--Zannier strategy for proving André--Oort}

We will outline the various ingredients used in Pila and Zannier's strategy for proving the André--Oort conjecture, and how \cref{height-disc-r-bound} fits into this.
There are several expositions available on how this strategy is implemented once one has the necessary ingredients (see, for example, \cite{pila-tsimerman:ao-surfaces} for the case of \( \cA_2 \), the moduli space of principally polarised abelian surfaces, \cite{ullmo:applications} for the case of \( \cA^n_g \) and \cite{daw:andre-oort} for the case of a general Shimura variety).

We begin by recalling the statement of the André--Oort Conjecture.

\begin{conjecture} \label{andre-oort}
Let $S$ be a Shimura variety and let $\Sigma$ be a set of special points in $S$. Every irreducible component of the Zariski closure of $\Sigma$ in $S$ is a special subvariety.
\end{conjecture}

The most appealing feature of the Pila--Zannier strategy is the manner in which it combines a number of independent ingredients to deliver a relatively simple proof of the conjecture.
The ingredients themselves are substantially more complicated and belong to various branches of mathematics.
The primordial result is the so-called Pila--Wilkie counting theorem (see~\cite{PW06}), yielding strong upper bounds on the number of algebraic points of bounded height and degree away from the algebraic part of a set definable in an o-minimal structure.

The fact that Shimura varieties are amenable to tools from o-minimality is due to the second ingredient, which states that, if we write $S:=\Gamma \backslash X^+$, the restriction of the uniformisation map $X^+ \rightarrow S$ to $\cF$ is definable in the o-minimal structure $\bR_{\rm an,exp}$. For the moduli space $\cA_g$ of principally polarised abelian varieties this is a theorem of Peterzil and Starchenko (see~\cite{PS10}). It has since been demonstrated for any Shimura variety in the work~\cite{kuy:ax-lindemann} of Klingler, Ullmo and Yafaev. 

The third ingredient is the so-called hyperbolic Ax-Lindemann-Weierstrass conjecture, which is a statement in functional transcendence regarding the uniformisation map. This was first proved for products of modular curves by Pila in his seminal work~\cite{P11}, then for compact Shimura varieties by Ullmo and Yafaev in~\cite{UY14c} and for $\cA_g$ by Pila and Tsimerman in~\cite{pila-tsimerman:ax-lindemann}.
The case of a general Shimura variety has recently been demonstrated by Klingler, Ullmo and Yafaev (see~\cite{kuy:ax-lindemann}). 

The final two ingredients are arithmetic in nature and serve as opposing forces.
The goal of these two ingredients is to show that there are constants $\newC{pw-bound-multiplier}$ and $\newC{pw-bound-exponent}$, such that for every pre-special point $x \in \cF$ with image $[x,1]$ in the Shimura variety,
\[ \rH(x) \leq \refC{pw-bound-multiplier} \, \abs{\Gal(\Qalg/L_x) \cdot [x,1]}^{\refC{pw-bound-exponent}}. \]
This is broken down into two parts.
The first is the lower bound for the sizes of Galois orbits of special points.
As we discussed previously, this is known unconditionally for $\Ag$ and under the GRH for all Shimura varieties.
Without the GRH, this bound remains an open problem in general.
The second of the final ingredients is the upper bound for the heights of pre-special points in the fundamental domain $\cF$, which is the subject of this paper.
We have already discussed previous work on special cases of these final two ingredients.

Gao has generalised the Pila--Zannier strategy to the mixed André--Oort conjecture (see \cite{Gao13}).
In the note \cite{Gao15}, he shows that \cref{height-disc-r-bound,mainindex} imply the mixed André--Oort conjecture for all mixed Shimura varieties, conditional on the same Galois bounds for pure Shimura varieties as \cref{andre-oort-conditional}.

\subsection{Pre-special points and realisations}

We emphasise that we are talking about the heights of \textit{pre-special} points, namely points belonging to $X^+$, rather than \textit{special points}, which in our terminology are precisely the images of pre-special points in $S$. Recall that $X$ is a $\gG(\bR)$-conjugacy class of morphisms $\bS\rightarrow G_{\bR}$, where $\bS$ is the Deligne torus. An element $x \in X$ is said to be \defterm{pre-special} if it factors through a subtorus of $\gG$ defined over $\bQ$. In the classical setting, where $X^+$ is \textit{realised} as the upper half-plane or the Siegel upper half-space $\cH_g$, points of $X^+$ are often referred to as \defterm{period matrices}. The period matrices corresponding to CM abelian varieties always have algebraic entries. The theorem of Pila and Tsimerman, which we are generalising, bounds the height of the period matrices of CM abelian varieties.

In general, in order to define the height \( \rH(x) \) of a point \( x \in X \), we must choose a \defterm{realisation} \( \cX \) of \( X \). By this we mean an analytic subset of a quasi-projective complex variety, equipped with a transitive action of \( \gG(\bR) \) by holomorphic automorphisms of \( \cX \) and with an isomorphism of \( \gG(\bR) \)-homogeneous spaces \( \cX \to X \) such that, for every \( x_0 \in \cX \), the map 
\[ \gG(\bR) \rightarrow \cX : g \mapsto g \cdot x_0 \]
is semi-algebraic (regarding \( \cX \) as a subset of a real algebraic variety by taking real and imaginary parts of the coordinates).
A morphism of realisations is defined to be a \( \gG(\bR) \)-equivariant biholomorphism commuting with the respective isomorphisms with \( X \).
By \cite{ullmo:applications} Lemme~2.1, every realisation is a semialgebraic set and every morphism of realisations is semialgebraic.

We restrict ourselves to realisations \( \cX \) for which the the action of \( \gG(\QalgR) \) is \defterm{semialgebraic over \( \QalgR \)}.
By this, we mean that:
\begin{enumerate}[(i)]
\item \( \cX \) is an analytic subset of the complex points of a quasi-projective variety defined over \( \Qalg \), and
\item for each \( \Qalg \)-point \( x_0 \in \cX \), the map \( g \mapsto g \cdot x_0 \) is definable in the language of ordered rings with constants \( \QalgR \).
(In other words, the graph of this map is a semialgebraic set which can be defined by equations and inequalities whose coefficients are in \( \QalgR \).)
\end{enumerate}
Note that throughout this paper, \( \Qalg \) denotes the subfield of algebraic numbers in~\( \bC \), so \( \QalgR \) is unambiguously defined.
The Borel realisation is an example of a realisation which is semialgebraic over \( \QalgR \) (see~\cite{UY11}) and any two such realisations are related by an isomorphism which is semialgebraic over \( \QalgR \).

By~\cite{UY11} Proposition~3.7, pre-special points in such a realisation have coordinates in \( \Qalg \).
Alternatively, this fact follows from the fact that each pre-special point in \( X \) is the unique fixed point of some element of \( \gG(\bQ) \) (see \cite{DH14} Theorem~2.3).
Therefore, we can write \( \rH(x) \) for a pre-special point \( x \in X \) to mean its multiplicative Weil height in a chosen realisation \( \cX \).
Henceforth, for any Shimura datum \( (\gG,X) \), we tacitly assume the choice of a realisation for \( X \).

In order for the Pila--Zannier strategy to work, it is necessary that pre-special points in \( \cX \) are not only algebraic but are defined over number fields of uniformly bounded degree.
In the case of the Borel realisation, it follows from the proof of~\cite{UY11} Lemma 3.8 that, given a faithful representation \( \rho \colon \gG \to \gGL_n \), a pre-special point \( x \in \cX \) is defined over the splitting field \( L \) of a maximal torus in \( \gGL_n \) containing the Mumford-Tate group of \( x \).
The rank \( d \) of this torus is of course bounded by \( n \), and the Galois action on its group of characters is given by an embedding
\[ \Gal(L/\bQ) \hookrightarrow \gGL_d(\bZ). \]
Therefore, it follows from a classical result of Minkowksi on finite subgroups of \( \gGL_d(\bZ) \) that the degree \( [L:\bQ] \) is bounded by a constant depending only on \( n \).
The fact that any two realisations semialgebraic over \( \QalgR \) are related by an isomorphism which is semialgebraic over \( \QalgR \) implies that this holds for all such realisations.

\subsection{Precise statement of our final bound and minor remarks}

For convenience, we include here a precise statement of our final bound for heights of pre-special points (i.e.\ the combination of \cref{height-disc-r-bound} and \cref{mainindex}).
The form of the bound in \cref{main-bound} matches the form of the conjectural Galois bounds (as mentioned in \cref{andre-oort-conditional}).
Hence the invariants which appear seem to be the natural invariants of a (pre-)special point of a general Shimura variety.

\begin{theorem} \label{main-bound}
Let \( (\gG, X) \) be a Shimura datum in which \( \gG \) is an adjoint group.

Choose a realisation of \( X \) such that the action of \( \gG(\QalgR) \) on \( X \) is semialgebraic and defined over \( \QalgR \).
Let \( \rH(x) \) denote the multiplicative Weil height of a point in \( X(\Qalg) \) with respect to this realisation.

Let \( K \) be a compact open subgroup of \(\gG(\bA_f) \) and let \( \cF \) be a fundamental set in \( X \) for \( K \cap\gG(\bQ) \), as in Théorème~13.1 of~\cite{borel:arithmetic-groups}, with respect to a pre-special base point \( x_0\). 

There exist constants \( \newC{main-bound-torus-exponent} \), \( \newC{main-bound-disc-exponent} > 0 \) such that for all \( \newC{main-bound-badprime-base} > 0 \), there exists \( \newC{main-bound-multiplier} > 0 \) (where \( \refC{main-bound-torus-exponent} \) and \( \refC{main-bound-disc-exponent} \) depend only on \( \gG \), \( X \), \( \cF \) and the realisation of \( X \), while \( \refC{main-bound-multiplier} \) depends on these data and also \( \refC{main-bound-badprime-base} \)) such that:

For each pre-special point \( x \in \cF \):
\begin{enumerate}[(a)]
\item Let \( \gM \) denote the Mumford--Tate group of \( x \) (which is a torus because \( x \) is pre-special).
\item Let \( K_{\gM} = K \cap \gM(\bA_f) \) and let \( K_{\gM}^m \) be the maximal compact subgroup of \( \gM(\bA_f) \).
\item Let \( i(\gM) \) be the number of primes \( p \) for which \( K \cap \gM(\bQ_p) \) is strictly contained in the maximal compact subgroup of \( \gM(\bQ_p) \).
\item Let \( L_x \) be the splitting field of \( \gM \).
\end{enumerate}

Then
\[ \rH(x) \leq \refC{main-bound-multiplier} \, \refC{main-bound-badprime-base}^{i(\gM)} [K_{\gM}^m:K_{\gM}]^{\refC{main-bound-torus-exponent}} \abs{\disc L_x}^{\refC{main-bound-disc-exponent}}. \]
\end{theorem}

The quantifiers associated with the constants in \cref{main-bound} are complicated.
In particular, why is \( \refC{main-bound-badprime-base} \) universally quantified while the others are existentially quantified?
The reason for this is the need to compare \cref{main-bound} with the Galois bound
\[ \refC{galois-bound-multiplier} \, \refC{galois-bound-badprime-base}^{i(\gM)} [K_{\gM}^m:K_{\gM}]^{\refC{galois-bound-torus-exponent}} \abs{\disc L_x}^{\refC{galois-bound-disc-exponent}} \leq \abs{\Gal(\Qalg/L_x) \cdot [x,1]} \]
from the assumption of \cref{andre-oort-conditional} and end up with a conclusion
\[
\rH(x) \leq \refC{pw-bound-multiplier} \, \abs{\Gal(\Qalg/L_x) \cdot [x,1]}^{\refC{pw-bound-exponent}}.
\tag{*} \label{eqn:pw-bound}
\]

The difficulty in comparing these two bounds is that \( \refC{galois-bound-badprime-base} \) might be less than~\( 1 \).
We will need to choose a large value for \( \refC{pw-bound-exponent} \) in order for the \( [K_{\gM}^m:K_{\gM}] \) and \( \abs{\disc L_x} \) factors on the right hand side of \eqref{eqn:pw-bound} to be larger than the corresponding factors on the left.
In order for the \( i(\gM) \) factors in \eqref{eqn:pw-bound} to work out, we need
\[
\refC{main-bound-badprime-base} \leq \refC{galois-bound-badprime-base}^{\refC{pw-bound-exponent}}.
\tag{\dag} \label{eqn:badprime-base-comparison}
\]
But if \( \refC{galois-bound-badprime-base} < 1 \), then we cannot achieve \eqref{eqn:badprime-base-comparison} by making \( \refC{pw-bound-exponent} \) large.
Instead we need the freedom to choose \( \refC{main-bound-badprime-base} \) in \cref{main-bound}.
Furthermore, our choice of \( \refC{main-bound-badprime-base} \) depends on \( \refC{pw-bound-exponent} \) which in turn depends on \( \refC{main-bound-torus-exponent} \) and \( \refC{main-bound-disc-exponent} \) so the quantifiers of these constants must be ordered as in the statement of the theorem.
Meanwhile, ``there exists \( \refC{main-bound-multiplier} \)'' has to come last due to the proof of \cref{main-bound}.

In the statements of \cref{height-disc-r-bound,main-bound} we assume that \( \gG \) is an adjoint group i.e.\ it has trivial centre.
In the context of the André--Oort conjecture this is entirely inconsequential (see~\cite{EY03} \S2).
The fact that \( \gG \) is adjoint ensures that the Hodge structures induced by any representation of \( \gG \) are of pure weight, which is essential to our proof (indeed, it ensures that these Hodge structures are pure of weight~\( 0 \), but this is not essential).
We have taken advantage of the hypothesis that \( \gG \) is adjoint to make some other, non-essential, simplifications.

Note that it is a trivial observation that \cref{height-disc-r-bound,main-bound} fail if one does not restrict to a fundamental set in \( X \).

\subsection{Comparison with Theorem~3.1 of \cite{pila-tsimerman:ao-surfaces}}

In essence, the proof of \cite{pila-tsimerman:ao-surfaces} Theorem~3.1 has three steps: studying the relationship between polarisations of a CM abelian variety and its endomorphism ring, choosing a suitable symplectic basis and reduction theory for matrices in the Siegel upper half-space \( \Hg \).
(The division between the cases of simple and non-simple abelian varieties which appears in \cite{pila-tsimerman:ao-surfaces} affects only the first of these steps.
It is true that our paper could be greatly simplified if we could ignore the case of non-irreducible isotypic Hodge structures, but given the final structure of our proof, we cannot easily separate it into parts dealing with irreducible and non-irreducible Hodge structures.)

We use polarisations in the same way as \cite{pila-tsimerman:ao-surfaces} (\cref{polarisation-bound}).
We also use reduction theory (section~\ref{ssec:reduction-theory}), although in dealing with general Shimura varieties, our calculations are necessarily less explicit.

The main additional difficulty for general Shimura varieties concerns the second of these steps.
Instead of a symplectic basis, we must find a basis having the property we call \textit{\( \gG \)-admissibility},
which we can describe using a multilinear form~\( \Phi \) on our Hodge structure whose stabiliser is equal to \( \gG \) (the existence of \( \Phi \) is guaranteed by Chevalley's theorem).
In the case of \( \Ag \), \( \gG = \gGSp_{2g} \) and this multilinear form is the standard symplectic form.
(Of course, \( \gGSp_{2g} \) is the group of similitudes of the standard symplectic form, rather than its stabiliser, but this is an unimportant technical difference.)
It is much easier to manipulate symplectic forms than general multilinear forms, and this leads to the most important difficulties in this paper (sections~\ref{ssec:g-admissible-basis} and~\ref{sec:variety-bound}).

It is perhaps worth saying a word about why this issue of \( \gG \)-admissibility is so important.
Throughout the paper, we work with a faithful representation \( \gG \hookrightarrow \gGL_n \) and we embed \( X \subset \Hom(\bS, \gG) \) in a \( \gGL_n(\bC) \)-conjugacy class \( X_{\gGL_n} \) of morphisms \( \bS_\bC \to \gGL_{n,\bC} \).
A basis is said to be \( \gG \)-admissible if it lies in the \( \gG(\bR) \)-orbit of some fixed reference basis.
Thus \( \gG \)-admissibility has two parts: the relevant matrix must be both real and in \( \gG \).
The requirement that this matrix be real is the purpose of section~\ref{sec:variety-bound} while the requirement that it be in \( \gG \) is the core of the difficulties in section~\ref{sec:bases}.

In earlier attempts to prove our theorem, we ignored \( \gG \)-admissibility and attempted to use reduction theory directly in \( \gGL_n(\bR) \) rather than in \( \gG(\bR) \).
This fails because reduction theory in \( \gGL_n(\bR) \) works with the symmetric space \( \gGL_n(\bR)/\bR^\times\gO_n(\bR) \), which is not the same as \( X_{\gGL_n} \).

\subsection{Outline of paper}

The proof of \cref{height-disc-r-bound} is found in section~\ref{sec:bases}.
In order to attach Hodge structures to points of~\( X \), we have to choose a representation of the group \( \gG \).
Of course, the constants we get depend on which representation we choose but the fact that there always exists some representation satisfying the conditions of \cref{height-disc-r-bound} means that this is sufficient for proving \cref{main-bound}.

In section~\ref{sec:variety-bound} we prove the following theorem, saying that an affine variety defined over \( \QalgR \) has a \( \QalgRb \)-point whose height is polynomially bounded relative to the coefficients of polynomials defining the variety.
A version of this with \( \QalgR \) replaced by \( \Qalg \) is straightforward, but we need the version with \( \QalgR \) in section~\ref{sec:bases}.

\begin{theorem} \label{variety-bound-intro}
For all positive integers \( m \), \( n \), \( D \), there are constants \( \newC{point-multiplier} \) and \( \newC{point-exponent} \) depending on \( m \), \( n \) and \( D \) such that:

For every affine algebraic set \( V \subset \bA^n \) defined over \( \QalgR \) by polynomials \( \seq{f}{m} \in \QalgRb[\seq{X}{n}] \) of degree at most \( D \) and height at most \( H \), if \( V(\bR) \) is non-empty, then \( V\QalgRb \) contains a point of height at most \( \refC{point-multiplier} H^{\refC{point-exponent}} \).
\end{theorem}

Most of this section is elementary, based on the proof of the Noether normalisation lemma.
However in some cases, in order to prove the \( \QalgR \) version of \cref{variety-bound-intro}, we need to show that if we have bounds for the heights of polynomials defining a variety \( V \), then we can find a proper algebraic subset \( V' \subset V\) such that \( \Sing V \subset V' \) and the heights of polynomials defining \( V' \) are also bounded.
The proof of this uses Philippon's arithmetic Bézout theorem (\cite{philippon}), Nesterenko's study of the Chow form (\cite{N77}) and an idea of Bombieri, Masser and Zannier (\cite{BMZ07}) to use the Chow form in studying the singular locus.

Finally in section~\ref{sec:discriminants} we relate the bound in terms of the discriminant of the centre of the endomorphism ring which appears in \cref{height-disc-r-bound} to the bound in terms of invariants of the Mumford--Tate torus which appears in \cref{main-bound}.
This generalises arguments of Tsimerman for the \( \Ag \) case (see~\cite{T12} section~7.2).

Let us say a little more about the proof of \cref{height-disc-r-bound}.
This proof is inspired by the proof of \cite{pila-tsimerman:ao-surfaces} Theorem~3.1 but turns out to be significantly more difficult, as explained above (and we have found it convenient to write it in a very different way).
We fix a pre-special base point \( x_0 \) for \( X \), and we aim to find an element (which we will call \( g_4 \)) in \( \gG(\QalgR) \) such that
\[ g_4 x_0 = x \text{ and the height of } g_4 \text{ is polynomially bounded}. \]

In order to construct \( g_4 \), we first construct several different elements of \( \gGL(E_\bC) \) which map \( x_0 \) to \( x \) (by conjugation in \( \Hom(\bS_\bC, \gGL(E_\bC)) \) but which do not satisfy all the conditions we want for \( g_4 \): they are not always in \( \gG(\bR) \) and they satisfy weaker bounds than a straightforward height bound.
(Note: in section~\ref{sec:bases}, we talk about bases for \( E_\bC \) rather than elements of \( \gGL(E_\bC) \).
These are equivalent once we have fixed a reference basis.)
These constructions use Minkowski's theorem, the theory of maximal tori and some calculations with Hermitian forms.

We then use \cref{variety-bound-intro} to construct \( g_4 \) itself.
Initially we do not get a height bound for \( g_4 \), only for the matrix relating \( g_3 \) and \( g_4 \).
However up to this point we have not used the fact that \( x \) is in the fundamental set~\( \cF \).
We use the definition of the fundamental set, together with our various other bounds for \( g_4 \), to conclude that the height of \( g_4 \) is bounded.

This postprint differs from the published version in the following ways.
Some explanatory paragraphs of the postprint are omitted in the published version.
The proof of \cref{root-bound} in the published version is incorrect.
It has been corrected in this postprint via the addition of Lemma~\ref{correction:um-roots}.

\section{Height bound in terms of the discriminant of the endomorphism ring} \label{sec:bases}

In this section we show that the heights of pre-special points in suitable fundamental sets are polynomially bounded with respect to the discriminant of the centre of the endomorphism ring of the associated Hodge structure.
In other words, we prove \cref{height-disc-r-bound}.

To prove \cref{height-disc-r-bound}, we construct an element \( g_4 \in \gG(\QalgR) \) such that
\[ g_4 x_0 = x \]
and \( g_4 \) has polynomially bounded height.
(Throughout this section, \defterm{polynomially bounded} means bounded above by an expression of the form \( C \, \abs{\disc R_x}^{C'} \).)
This suffices to prove the theorem because the action of \( \gG(\bR) \) on \( X \) is semialgebraic.

On the way to constructing \( g_4 \), we will need to talk about \( gx_0 \) not just for \( g \in \gG(\bR) \) but also for \( g \in \gGL(E_\bC) \).
In order for this to make sense, we enlarge
\[ X = \text{ a } \gG(\bR) \text{-conjugacy class in } \Hom(\bS, \gG_\bR) \]
to a \( \gGL(E_\bC) \)-conjugacy class in \( \Hom(\bS_\bC, \gGL(E_\bC)) \).
Whenever we write \( gx_0 \) for \( g \in \gGL(E_\bC) \) we always refer to the action of \( \gGL(E_\bC) \) on \( \Hom(\bS_\bC, \gGL(E_\bC)) \) by conjugation.

Rather than talking about elements of \( \gGL(E_\bC) \), we will often talk about bases for \( E_\bC \).
Since \( \gGL(E_\bC) \) acts simply transitively on the set of bases, this is purely a matter of language.
It is more convenient to use the language of bases because, while we are ultimately interested in how a general basis \( \cB \) is related to \( \basis{0} \) (a fixed basis linked to the base point \( x_0 \)), for our computations we shall want to consider the coordinates of \( \cB \) relative to \( \cB_\bZ \) (a fixed basis for the lattice \( \cB_\bZ \)).

We shall define a number of adjectives to describe special types of (ordered) bases for \( E_\bC \).
For now, we omit certain technical complications from these definitions;
the full details will appear in section~\ref{sssec:basis-types}.

\begin{enumerate}
\item A basis \( \cB \) for \( E_\bC \) is a \defterm{Hodge basis} for \( x \in X \) if every vector in \( \cB \) is an eigenvector for the Hodge parameter \( x \) (and \( \cB \) is ordered in a correct way).

The purpose of this definition is to relate bases to elements of \( x \in X \): \( \cB \) will be a Hodge basis for \( x \) if and only if the element \( g \in \gGL(E_\bC) \) mapping \( \basis{0} \) to \( \cB \) satisfies \( gx_0 = x \).

\item A basis \( \cB \) for \( E_\bC \) is a \defterm{diagonal basis} for \( E_x \) if each vector in \( \cB \) is an eigenvector for the centre of \( \End_{\bQ\text{-HS}} E_x \), and \( \cB \) is a Hodge basis.
(If \( x \) is pre-special, then the condition that \( \cB \) is a Hodge basis serves only to control the ordering of \( \cB \)).

\item A basis \( \cB \) for \( E_\bC \) is \defterm{\( \gG \)-admissible} if the element \( g \in \gGL(E_\bC) \) which maps \( \basis{0} \) to \( \cB \) is in \( \gG(\bR) \).

\item A diagonal basis \( \cB \) for \( E_x \) is \defterm{Galois-compatible} if the action of \( \Aut(\bC) \) on coordinates permutes the vectors of \( \cB \) in a suitable way.

The benefit of Galois-compatibility is that, if we have a Galois-compatible basis and we can bound the absolute values and the denominators of its coordinates, then we can deduce a bound for the heights of the coordinates.

\item A basis \( \cB \) for \( E_\bC \) is \defterm{weakly bounded} if all its coordinates are algebraic numbers with polynomially bounded denominators and its covolume is polynomially bounded.
\end{enumerate}

Using these adjectives, we can translate our search for \( g_4 \) into the following: we construct a \( \gG \)-admissible Hodge basis \( \basis{4} \) for the pre-special point \( x \) whose coordinates are algebraic numbers with polynomially bounded height.
In order to do this, we will first construct a series of other bases satisfying some but not all of the properties we want:

\begin{enumerate}
\item \( \basis{1} \), a \( \gG \)-admissible diagonal basis with no bounds at all;

\item \( \basis{2} \), a Galois-compatible weakly bounded diagonal basis which need not be \( \gG \)-admissible;

\item \( \basis{3} \), a Galois-compatible weakly bounded diagonal basis on which the values of a polarisation are bounded;

\item \( \basis{4} \), a weakly bounded \( \gG \)-admissible diagonal basis, such that the coordinates of \( \basis{4} \) with respect to \( \basis{3} \) are algebraic numbers of polynomially bounded height.
\end{enumerate}

Once we have obtained \( \basis{4} \) as above, we will then prove that \( \basis{4} \) has coordinates of polynomially bounded height (relative to \( \cB_\bZ \)).
The point of bounding the height of the coordinates of \( \basis{4} \) relative to \( \basis{3} \) is that it allows us to pass bounds back and forth between \( \basis{3} \) and \( \basis{4} \).
In particular this means that we can exploit the Galois-compatibility of \( \basis{3} \) while bounding \( \basis{4} \) (which is not Galois-compatible).

The constructions of \( \basis{1} \), \( \basis{2} \) and \( \basis{3} \) are fairly straightforward. \( \basis{1} \) is constructed geometrically using maximal tori, while \( \basis{2} \) is constructed arithmetically using the Minkowski bound.
\( \basis{3} \) is obtained from \( \basis{2} \) by some calculations with positive-definite Hermitian forms.

The hardest part of the proof is to obtain \( \basis{4} \) from \( \basis{3} \).
To do this we will use Chevalley's theorem: we can choose a multilinear form \( \Phi \colon E_\bZ^{\otimes k} \to \bZ \) such that
\[ \gG = \{ g \in \gGL(E) \mid \Phi(gv_1, \dotsc, gv_k) = \Phi(v_1, \dotsc, v_k) \text{ for all } v_j \in E \}. \]
Then a basis \( \cB \) is \( \gG \)-admissible if and only if the values of \( \Phi \) on \( \cB \) are the same as its values on \( \basis{0} \).

We can use the bound for the polarisation on \( \basis{3} \) together with the Galois-compatibility of \( \basis{3} \) to show that the values of the multilinear form \( \Phi \) on \( \basis{3} \) are algebraic numbers of polynomially bounded height.
This allows us to conclude that the following algebraic set is defined by equations of polynomially bounded height:
\[ V = \{ h \in \gGL_n(\bC) \mid \basis{3} h \text{ is diagonal and } \gG\text{-admissible} \}. \]
Here we write \( \basis{3} h \) with \( h \) on the right to mean the basis whose coordinates relative to \( \basis{3} \) are given by the columns of \( h \). This is different from \( g\basis{3} \) with a linear map \( g \) on the left, which means the basis obtained by applying \( g \) to each element of \( \basis{3} \).

We can apply \cref{variety-bound} to \( V \) to obtain \( \basis{4} \) (the existence of \( \basis{1} \) ensures that \( V(\bR) \) is non-empty).

Finally we have to bound the height of \( \basis{4} \).
In order to do this we will use the fact that \( x \) is in the fundamental domain \( \cF \) -- this fact has not been used so far in the construction of \( \basis{4} \).
In particular, this means that the element \( g_4 \in \gG(\bR) \) mapping \( \basis{0} \) to \( \basis{4} \) is contained in a Siegel set.
Using the definition of Siegel sets together with bounds for \( \basis{4} \) deduced from the fact that \( \basis{3} \) is weakly bounded and Galois-compatible, we prove that the coordinates of \( \basis{4} \) have polynomially bounded absolute values.
Using again the height bound for the relation between \( \basis{3} \) and \( \basis{4} \) and the fact that \( \basis{3} \) is Galois-compatible, we deduce that the coordinates of \( \basis{4} \) have polynomially bounded height.

\subsection{Notation}

In the first three parts of this section, we define notation for data which depend only on the Shimura datum \( (\gG, X) \) and the representation \( \rho \) (when we say ``choose'' something, we choose it once depending on \( (\gG, X, \rho) \) and thereafter regard it as fixed).  In particular these data do not depend on the special point \( x \).
In section~\ref{sssec:variable-data}, we define notation for data which does depend on \( x \).

Choose a basis \( \cB_\bZ \) for \( E_\bZ \).
Whenever we talk about the coordinates of a basis in \( E_\bC \), we mean with respect to \( \cB_\bZ \).

Let \( n = \dim E \).

If \( \cB \) and \( \cB' \) are two bases of the vector space \( E \), we define
\[ \Vol(\cB : \cB') = \operatorname{covol}(\cB')/\operatorname{covol}(\cB). \]

\subsubsection{Multilinear forms}

By Chevalley's theorem, we can choose a multilinear form \( \Phi \colon E_\bZ^{\otimes k} \to \bZ \) such that
\[ \gG = \{ g \in \gGL(E) \mid \Phi(gv_1, \dotsc, gv_k) = \Phi(v_1, \dotsc, v_k) \text{ for all } v_j \in E \}. \]
(By Deligne's extension of Chevalley's theorem, because \( \gG \) is reductive, we can require that \( g \) exactly preserves \( \Phi \) rather than just up to a scalar.
Because of the hypothesis that \( \rho \) is self-dual, we can say that \( \Phi \) is in \( E_\bZ^{\vee \otimes k} \) rather than a mixture of \( E_\bZ^\vee \) and \( E_\bZ  \).)

Let \( \Psi \colon E_\bQ \times E_\bQ \to \bQ \) denote a \( \gG \)-invariant \( C \)-polarisation as in \cite{deligne:shimura-vars}~1.1.15 and~1.1.18(b).

By the definition of a \( C \)-polarisation, for each \( x \in X \), the bilinear form on \( E_\bC \) defined by
\[ \Theta_x(u, v) = \Psi_\bC(\bar{u}, x(i) v) \]
is Hermitian and positive definite, where \( \Psi_\bC \) denotes the \( \bC \)-bilinear extension of~\( \Psi \).

\subsubsection{Base point}

Choose a pre-special point \( x_0 \in \cF \) such that the fundamental set \( \cF \) is a finite union of \( \gG(\bQ) \)-translates of \( \fS.x_0 \), where \( \fS \) is a Siegel set for \( \gG \) over \( \bQ \).

Let \( \gT_0 \) be a maximal torus of \( \gG \) defined over \( \QalgR \) which contains the image of the Hodge parameter \( x_0 \).

Choose a basis \( \basis{0} \) for \( E_\bC \) which satisfies the following conditions:
\begin{enumerate}[(a)]
\item every vector in \( \basis{0} \) is an eigenvector for \( \gT_0 \);
\item for each \( v \in \basis{0} \), the complex conjugate \( \bar{v} \) is also in \( \basis{0} \);
\item each vector of \( \basis{0} \) has coordinates in \( \Qalg \) (relative to \( \cB_\bZ \)).
\end{enumerate}

Note that each eigenspace of \( \gT_0 \) is contained in one of the Hodge components \( E_{x_0}^{p,-p} \), so condition (a) implies that \( \basis{0} \) consists of eigenvectors for the Hodge parameter \( x_0 \).
In condition (b), \( \bar{v} \) may be equal to \( v \).

To construct such a basis \( \basis{0} \), observe that each eigenspace of \( \gT_0 \) is defined over \( \Qalg \), and that the complex conjugate of an eigenspace of \( \gT_0 \) is also eigenspace of \( \gT_0 \) (perhaps the same one).
So we simply take the union of the following:

\begin{enumerate}[(i)]
\item for each pair of distinct complex conjugate eigenspaces \( E_\chi, E_{\bar\chi} \) of \( \gT_0 \), choose one of these eigenspaces \( E_\chi \) and choose a basis for \( E_\chi \) defined over \( \Qalg \);
\item the complex conjugates of (i);
\item for each real eigenspace of \( \gT_0 \), choose any basis for the eigenspace which is defined over \( \QalgR \).
\end{enumerate}

\subsubsection{Variable data} \label{sssec:variable-data}

Let \( F_x \) denote the centre of \( \End_{\bQ\text{-HS}} E_x \)
and \( R_x \) the centre of \( \End_{\bZ\text{-HS}} E_x \).
We can decompose \( F_x \) as a direct product of fields:
\[ F_x = \prod_{i=1}^s F_i \]
where each \( F_i \) is either a CM field or \( \bQ \) (\( \bQ \) appears if \( E_x^{0,0} \neq \{ 0 \} \)).

Let \( \Sigma = \Hom(F_x, \bC) \), and let \( \Sigma_i = \Hom(F_i, \bC) \) for \( 1 \leq i \leq s \).
Then \( \Sigma \) is the disjoint union of the \( \Sigma_i \).

For each \( \sigma \in \Sigma \), let \( r_\sigma \) denote the dimension of the eigenspace in \( E_x \) on which \( F_x \) acts via \( \sigma \).
Then \( r_\sigma \) is the same for all \( \sigma \) in the same \( \Sigma_i \), and we also denote this value \( r_i \).

\subsubsection{Types of basis} \label{sssec:basis-types}

We will define a number of adjectives which we will use to describe bases of \( E_\bC \).
Implicitly, our bases will be ordered, so that given two bases \( \cB \) and \( \cB' \), the ``linear map \( g \) such that \( g\cB' = \cB \)'' is well-defined.
However we will never explicitly write down the ordering of a basis except during the definition of a Hodge basis -- indeed, most of our bases will be diagonal bases labelled as explained in the definition of a diagonal basis.

Choose and fix an ordering of the basis \( \basis{0} \) which we have already chosen.

\begin{enumerate}
\item An ordered basis \( \cB = \{ v_1, \dotsc, v_n \} \) of \( E_\bC \) is a \defterm{Hodge basis} for \( E_x \) if:
\begin{enumerate}
\item every vector in \( \cB \) is an eigenvector for the Hodge parameter \( x \); and
\item the Hodge type (in \( E_x \)) of each \( v_j \) is the same as the Hodge type (in~\( E_{x_0} \)) of the corresponding \( \bv{0}{j} \in \basis{0} \).
\end{enumerate}

Equivalently, a basis \( \cB \) is a Hodge basis for \( E_x \) if and only if the linear map \( g \in \gGL(E_\bC) \) defined by \( g\basis{0} = \cB \) satisfies
\[ gx_0 = x \text{ in } \Hom(\bS_\bC, \gGL(E_\bC)). \]

\item An ordered basis \( \cB \) of \( E_\bC \) is a \defterm{diagonal basis} for \( V_x \) if it is a Hodge basis and every vector in \( \cB \) is an eigenvector for the centre of \( \End_{\bQ\text{-HS}} E_x \).

If \( x \) is pre-special (which is the only case we care about), every eigenvector of the centre of \( \End_{\bQ\text{-HS}} E_x \) is automatically an eigenvector for \( x \), so requiring that \( \cB \) be a Hodge basis only adds conditions (b), that \( \cB \) is ordered so that its Hodge types match those of \( \basis{0} \).

We shall label the elements of a diagonal basis as
\[ \{ v_{\sigma j} \mid \sigma \in \Sigma, 1 \leq j \leq r_\sigma \} \]
such that \( F_x \) acts on \( v_{\sigma j} \) via the character \( \sigma \).

\item A basis \( \cB \) for \( E_\bC \) is \defterm{\( \gG \)-admissible} if the linear map \( g \in \gGL(E_\bC) \) defined by \( g\basis{0} = \cB \) satisfies
\[ g \in \rho(\gG(\bR)). \]
Note that this imposes two conditions on \( g \): that it is in the image of \( \gG \), and that it is real.

\item A diagonal basis \( \cB = \{ v_{\sigma j} \} \) for \( E_x \) is \defterm{Galois-compatible} if
\[ \tau(v_{\sigma j}) = v_{(\tau\sigma) j} \]
for all \( \tau \in \Aut(\bC) \)
(where the LHS means: apply \( \tau \) to each coordinate of \( v_{\sigma j} \) with respect to \( \cB_\bZ \)).

Equivalently, for each \( i \in \{ 1, \dotsc, s \} \), there are vectors
\[ w_{i1}, \dotsc, w_{ir_i} \in F_i^n \]
\( i \) such that
\[ v_{\sigma j} = \sigma(w_{ij}) \]
for all \( \sigma \in \Sigma_i \) and \( j \in \{ 1, \dotsc, r_i \} \).

We may also apply the term Galois-compatible to other collections indexed by \( \Sigma \), with the obvious meaning.

\item A basis \( \cB \) for \( E_\bC \) is \defterm{weakly bounded} if all its coordinates (with respect to \( \cB_\bZ \)) are algebraic numbers with polynomially bounded denominators and its covolume is polynomially bounded.
\end{enumerate}

\subsubsection{The Siegel set}

The notation we define here for the Siegel set will be used only in section~\ref{ssec:reduction-theory}.

The definition of Siegel set is taken from \cite{borel:arithmetic-groups}~12.3, but we have reversed the order of multiplication in \( \gG \) so that it has a left action on \( X \) instead of a right action.
Consequently we also reverse the inequality in the definition of \( A_t \).

Note also that for Borel, \( \gG_\bR^0 \) means the identity component of \( \gG(\bR) \) in the real topology, which we denote by \( \gG(\bR)^+ \), while Borel uses \( \gG^0 \) to mean the identity component of the algebraic group \( \gG \) in the Zariski topology, which we denote \( \gG^\circ \).

By definition, the Siegel set \( \fS \) has the form
\[ \fS = \omega A_t K \]
where
\begin{enumerate}[(i)]
\item \( K \) is the stabiliser of \( x_0 \) in \( \gG(\bR) \),
\item \( A_t = \{ a \in \gS(\bR)^+ \mid \varphi(a) \geq t \text{ for all simple } \bQ\text{-roots } \varphi \text{ of } \gG \text{ w.r.t.\ } \gU \} \) for some \( t \in \bR_{>0} \),
\item \( \omega \) is a compact neighbourhood of the identity in \( \gU(\bR)\gM(\bR)^+ \),
\item \( \gS \) is a maximal \( \bQ \)-split torus in \( \gG \),
\item \( \gM \) is the maximal \( \bQ \)-anisotropic subgroup of \( Z_\gG(\gS)^\circ \), and
\item \( \gU \) is the unipotent radical of a minimal \( \bQ \)-parabolic subgroup of \( \gG \) containing~\( \gS \).
\end{enumerate}

Without loss of generality, we may assume that \( \rho(\gS) \) is contained in the set of diagonal matrices and \( \rho(\gU) \) in the set of upper triangular matrices of \( \gGL(E) \) with respect to \( \cB_\bZ \) -- this simply requires us to replace \( \cB_\bZ \) by a fixed \( \gGL(E_\bQ) \)-conjugate, which makes a bounded change to \( \abs{\disc R_x} \).

By \cite{borel:arithmetic-groups}~13.1, our fundamental set \( \cF \) has the form
\[ \cF = C.\fS.x_0 \]
for some finite set \( C \subset \gG(\bQ) \).

\subsection{A \hgG-admissible diagonal basis}

We prove that there exists a \( \gG \)-admissible diagonal basis for \( E_x \).
First we construct \( g_1 \in \gG(\bR) \) which maps \( x_0 \) to \( x \) and which conjugates the fixed maximal torus~\( \gT_0 \) to a maximal torus containing the Mumford--Tate group of \( x \).
Then we let \( \basis{1} = g_1 \basis{0} \).
This is automatically a \( \gG \)-admissible Hodge basis which diagonalises the Mumford--Tate group of \( x \), and we prove that this implies that it is a diagonal basis.

\begin{lemma} \label{conjugate-mt-group}
There exists \( g_1 \in \gG(\bR) \) such that \( g_1x_0 = x \) and \( g_1 \gT_0 g_1^{-1} \) contains the Mumford--Tate group of \( E_x \).
\end{lemma}

\begin{proof}
First choose any \( g \in \gG(\bR) \) such that \( gx_0 = x \).
Then \( g \gT_0 g^{-1} \) is a maximal \( \bR \)-torus of \( \gG \) containing the image of \( x \), but it might not contain the Mumford--Tate group of \( E_x \).

Choose any maximal torus \( \gT \subset \gG \) defined over \( \bR \) which contains the Mumford--Tate group of \( E_x \).
Then \( g \gT_0 g^{-1} \) and \( \gT \) are both maximal \( \bR \)-tori in \( K_x \), the stabiliser of \( x \) in \( \gG(\bR) \).
Since \( K_x \) is compact, there is some \( g' \in K_x \) which conjugates \( g \gT_0 g^{-1} \) into \( \gT \).

Taking \( g_1 = g'g \) proves the lemma (\( g_1 x_0 = x \) because \( g' \) stabilises \( x \)).
\end{proof}

\begin{lemma} \label{mt-to-diagonal-basis}
Let \( \cB \) be a Hodge basis for \( E_x \).
If \( \cB \) diagonalises the Mumford--Tate group of \( E_x \),
then \( \cB \) is a diagonal basis for \( E_x \).
\end{lemma}

\begin{proof}
Let \( \gM \subset \gG \) be the Mumford--Tate group of \( E_x \).

We write \( \End \rho_{|\gM} \) for the endomorphism algebra of the \( \bQ \)-representation \( \rho_{|\gM} \colon \gM \to \gGL(E) \), considered as a subalgebra of \( \End(E_\bQ) \).
It is a standard property of Mumford--Tate groups that
\[ \End \rho_{|\gM} = \End_{\bQ\text{-HS}} E_x \]
and hence the centre of \( \End \rho_{|\gM} \) is \( F_x \).

Let \( \gZ \) denote the centre of the centraliser of \( \gM \) in \( \gGL(E) \).
Since the \( \bQ \)-points of the centraliser of \( \gM \) in \( \gGL(E) \) are the same as the invertible elements of \( \End \rho_{|\gM} \), we deduce that \( \gZ(\bQ) = F_x^\times \) (as subgroups of \( \gGL(E_\bQ) \)).

Since \( \gM \) is commutative, \( \gM \subset \gZ \) and so each eigenspace for \( \gZ \) in \( E_\bC \) is contained in an eigenspace for \( \gM \) in \( E_\bC \).
We need to establish the inclusion of eigenspaces in the opposite direction.

The isotypic decomposition of the \( \bQ \)-Hodge structure \( E_x \) is
\[ E_x = \bigoplus\nolimits_i E_{x,i}  \tag{*} \label{eqn:isotypic-decomposition} \]
where \( E_{x,i} \) is isomorphic to the \( r_i \)-th power of an irreducible \( \bQ \)-Hodge structure with endomorphism algebra \( F_i \).
Because sub-\( \bQ \)-Hodge structures are the same as subrepresentations of the Mumford--Tate group, this is the same as the isotypic decomposition of the \( \bQ \)-representation \( \rho_{|\gM} \).
Since \( \gM \) is a torus, it follows that each eigenspace of \( \gM \) in \( E_\bC \) is contained in one of the \( E_{x,i} \) and has dimension \( r_i \).

By looking at the action of \( \End \rho_{|\gM} \), we see that \eqref{eqn:isotypic-decomposition} is also the isotypic decomposition of the \( \bQ \)-representation \( \rho_{|\gZ} \), and hence that every eigenspace of \( \gZ \) in \( E_\bC \) is contained in one of the \( E_{x,i} \) and has dimension \( r_i \).

Thus each eigenspace of \( \gZ \) in \( E_\bC \) has the same dimension \( r_i \) as the eigenspace of \( \gM \) which contains it, so the eigenspaces of \( \gZ \) are equal to the eigenspaces of \( \gM \).
Thus the fact that \( \cB \) diagonalises \( \gM \) implies that it also diagonalises \( \gZ \) (or in other words, it diagonalises the action of \( F_x \) on \( E_x \)).
\end{proof}

\begin{proposition} \label{admissible-diagonal-basis}
There exists a \( \gG \)-admissible diagonal basis \( \basis{1} = \{ \bv{1}{\sigma j} \}\) for~\( E_x \).
\end{proposition}

\begin{proof}
Choose \( g_1 \) as in \cref{conjugate-mt-group}, and let \( \basis{1} = g_1 \basis{0} \).
Since \( g_1 x_0 = x \), this gives a Hodge basis for \( E_x \).

Since \( \basis{0} \) diagonalises \( \gT_0 \), \( \basis{1} \) diagonalises \( g_1 \gT_0 g_1^{-1} \) and hence a fortiori \( \gM \).
\Cref{mt-to-diagonal-basis} implies that \( \basis{1} \) is a diagonal basis for \( E_x \).

Finally \( \basis{1} \) is \( \gG \)-admissible because \( g_1 \in \gG(\bR) \).
\end{proof}

\subsection{A weakly bounded diagonal basis}

In order to prove the existence of a weakly bounded diagonal basis, we need a generalisation of Minkowski's bound to arbitrary lattices in a module over a product of number fields.
This is essentially Claim~3.1 from \cite{pila-tsimerman:ao-surfaces} (the generalisation of Minkowski's bound from ideals in a ring of integers to lattices), but we have included all the generalisations we will need in a single statement: working in a vector space over a number field rather than in the field itself, and working with  a product of number fields rather than just a single field.

The deduction of \cref{weakly-bounded-basis} from \cref{minkowski} is essentially just bookkeeping.
The construction gives us Galois-compatibility for free.

\begin{lemma} \label{minkowski}
For all positive integers \( d \), \( s \) and \( r_1, \dotsc, r_s \), there are constants \( \newC{minkowski-multiplier} \) and~\( \newC{minkowski-exponent} \) depending on \( d \), \( s \) and \( r_1, \dotsc, r_s \) such that:

For all number fields \( F_1, \dotsc, F_s \) of degree at most \( d \) and all \( \bZ \)-lattices
\[ L \subset \prod_i F_i^{r_i}, \]
if we let \( F = \prod_i F_i \), \( R = \Stab_F L \) and \( L_0 = \prod_i \cO_{F_i}^{r_i} \),
then there exists \( \nu \in \prod_i \gGL_{r_i}(F_i) \) such that
\[ \nu L \subset L_0 \text{ and } [L_0:\nu L] \leq \refC{minkowski-multiplier} \, \abs{\disc R}^{\refC{minkowski-exponent}}. \]
\end{lemma}

\begin{proof}
Let \( \cO_F = \prod_i \cO_{F_i} \subset F \),
and let \( e_R = [\cO_F:R] \).
Then
\[ \disc R = e_R^2 \prod_i \disc F_i. \]

Let \( L' = \cO_F.L \).
Then \( L' \) is an \( \cO_F \)-module such that
\[ e_R L' \subset L \subset L'. \]
This implies that \( [L':L] \leq e_R^{\rk_\bZ L} \).

Since \( L' \) is an \( \cO_F \)-module, it splits as a direct sum
\[ L' = \bigoplus_i L'_i \]
where each \( L'_i \subset F_i^{r_i} \) is an \( \cO_{F_i} \)-module of rank \( r_i \).
Furthermore since \( \cO_{F_i} \) is a Dedekind domain, \( L'_i \) is isomorphic (as an \( \cO_{F_i} \)-module) to \( \cO_{F_i}^{r_i-1} \oplus J_i \) for some ideal \( J_i \subset \cO_{F_i} \).
By Minkowski's first theorem, we can choose \( J_i \) such that \( [\cO_{F_i}:J_i] \) is bounded above by \( \abs{\disc F_i}^{1/2} \) times a constant depending only on \( d \).

In other words, there exists \( \nu_i \in \gGL_{r_i}(F_i) \) such that \( \nu_i L'_i = \cO_{F_i}^{r_i-1} \oplus J_i \), which is contained in \( \cO_{F_i}^{r_i} \) with index bounded by a constant times \( \abs{\disc F_i}^{1/2} \).

Letting \( \nu = (\nu_1, \dotsc, \nu_s) \in \prod_i \gGL_{r_i}(F_i) \), we get that
\[ \nu L \subset \nu L' \subset L_0 \text{ and } [L_0:\nu L] = [L':L] \prod_i [\cO_{F_i}^{r_i}:\nu L_i'] \]
and these indices are bounded relative to \( \abs{\disc R} \) as required.
\end{proof}

\begin{proposition} \label{weakly-bounded-basis}
There exists a Galois-compatible weakly bounded diagonal basis \( \basis{2} = \{ \bv{2}{\sigma j} \} \) for \( E_x \).
\end{proposition}

\begin{proof}
Choose an isomorphism of \( F_x \)-modules \( \iota \colon \prod_i F_i^{r_i} \to E_\bQ \).
Choose \( \nu \) as in \cref{minkowski} applied to \( L = \iota^{-1}(E_\bZ) \), and replace \( \iota \) by \( \iota \circ \nu^{-1} \).
Then
\[ E_\bZ \subset \iota(L_0) \text{ and } [\iota(L_0):E_\bZ] \text{ is polynomially bounded}. \]

There is a natural isomorphism of \( \bC \)-vector spaces
\[ \prod_i F_i^{r_i} \otimes_\bQ \bC \to \bC^n \]
built using all the embeddings \( F_i \to \bC \) on each copy of \( F_i \).
Let \( \cB_F \) be the basis for \( \prod_i F_i^{r_i} \) which maps to the standard basis of \( \bC^n \) under this isomorphism, and let
\[ \basis{2} = \iota(\cB_F). \]
This consists of eigenvectors for the action of \( F_x \) on \( E_x \).
If we order the vectors of~\( \basis{2} \) appropriately, then it will satisfy the condition on Hodge types to be a Hodge basis for \( E_x \), giving us a diagonal basis for \( E_x \).

Let \( P \) denote the matrix giving the coordinates of \( \cB_\bZ \) with respect to \( \basis{2} \), in other words
\[ e_{j'} = \sum_{\sigma \in \Sigma} \sum_{j=1}^{r_\sigma} P_{\sigma jj'} \bv{2}{\sigma j} \]
where the rows of \( P \) are indexed by \( (\sigma, j) \in \bigcup_{i=1}^s \Sigma_i \times \{ 1, \dotsc, r_i \} \) and the columns are indexed by \( j' \in \{ 1, \dotsc, n \} \).

Because \( \cB_\bZ \subset \iota(L_0) \), the entries of \( P \) are algebraic integers and \( \Aut(\bC) \) permutes the rows of \( P \) in the natural way i.e.
\[ \tau(P_{\sigma jj'}) = P_{(\tau\sigma) jj'}. \]
Furthermore
\[ \det P = \Vol(\basis{2}:\cB_\bZ) = \Vol(\cB_F:L_0) \Vol(L_0:E_\bZ) = \bigl(\prod_i \abs{\disc F_i }^{r_i/2} \bigr) [L_0:E_\bZ] \]
so \( (\det P)^2 \) is a polynomially bounded rational integer.

Hence \( P^{-1} \) has entries which are algebraic numbers with polynomially bounded denominators and its determinant is polynomially bounded (indeed \( \det P^{-1} \leq 1 \)).
Since \( P^{-1} \) gives the coordinates of \( \basis{2} \) with respect to \( \cB_\bZ \), this says that \( \basis{2} \) is weakly bounded.

Since \( \Aut(\bC) \) permutes the rows of \( P \), it permutes the columns of \( P^{-1} \) in the natural way, and thus \( \basis{2} \) is Galois-compatible.
\end{proof}

\subsection{A weakly bounded diagonal basis with polynomially bounded polarisation}

In order to construct \( \basis{3} \), we perform a long calculation with Hermitian forms.
The key step is Lemma~3.5 from \cite{pila-tsimerman:ao-surfaces}.
Essentially we want to apply this lemma to the values of the positive definite Hermitian form
\[ \Theta_x(u, v) = \Psi_\bC(\bar{u}, x(i)v) \]
on \( \basis{2} \), but we have to tweak these values slightly (by using \( \zeta \) constructed using \cref{totally-real-signs}) to get a set of values which are Galois-compatible.

\begin{lemma} \label{totally-real-signs}
For all positive integers~\( d \), there exist constants~\( \newC{sign-height-bound-multiplier} \) and~\( \newC{sign-height-bound-exponent} \) such that:

For every totally real number field \( F \) of degree \( d \) and every collection of signs \( s_\sigma \in \{ \pm 1 \} \) indexed by the embeddings \(\sigma \colon F \to \bR \), there exists an element \( \zeta \in F^\times \) such that
\[ \rH(\zeta) \leq \refC{sign-height-bound-multiplier} \abs{\disc F}^{\refC{sign-height-bound-exponent}} \]
and for each \( \sigma \colon F \to \bR \), the sign of \( \sigma(\zeta) \) is \( s_\sigma \).
\end{lemma}

\begin{proof}
Consider the lattice of integers \( \cO_F \) in \( F \otimes_\bQ \bR \cong \bR^d \).
Minkowski's second theorem implies that the covering radius \( \mu \) of this lattice is bounded by a polynomial in \( \abs{\disc F} \) (with the polynomial depending only on \( d \)).

Consider a ball \( B \) in \( F \otimes_\bQ \bR \) which has radius greater than \( \mu \) and which is contained in the hypercube
\[ \{ z \in F \otimes_\bQ \bR \mid 0 < s_\sigma \sigma(z) < 3\mu \text{ for all } \sigma \colon F \to \bR \}. \]

By the definition of the covering radius, \( B \) contains an element \( \zeta \) of the lattice~\( \cO_F \).
By our choice of \( B \), \( \sigma(\zeta) \) has the correct sign for each \( \sigma \colon F \to \bR \).
Since \( \zeta \) is an algebraic integer, its height satisfies
\[ \rH(\zeta) \leq \prod_{\sigma \colon F \to \bR} \abs{\sigma(\zeta)} < (3\mu)^d \]
which is bounded by a polynomial in \( \abs{\disc F} \) as required.
\end{proof}

\begin{lemma} \label{det-hermitian-bounded}
The determinant of the Hermitian matrix
\[ \Theta_x(\basis{2}, \basis{2}) \]
is polynomially bounded.
\end{lemma}

\begin{proof}
We have
\[ \det \Theta_x(\basis{2}, \basis{2}) = \abs{\Vol(\basis{1}:\basis{2})}^2 \det \Theta_x(\basis{1}, \basis{1}). \]

Now
\[ \Vol(\basis{0}:\basis{1}) = 1 \]
because \( \basis{1} \) is \( \gG \)-admissible and \( \gG \) is semisimple (so \( \det \rho(\gG) \) must be trivial).

Hence
\[ \Vol(\basis{1}:\basis{2}) = \Vol(\basis{0}:\basis{2}) \]
is polynomially bounded because \( \basis{2} \) is weakly bounded.
Also
\[ \Theta_x(\basis{1}, \basis{1}) = \Theta_{x_0}(\basis{0}, \basis{0}) \]
is constant, so the lemma is proved.
\end{proof}

The following lemma is proved in the same way as Lemma~3.5 of \cite{pila-tsimerman:ao-surfaces}, but for convenience we have written it down as a general statement with the precise list of conditions required.

\begin{lemma} \label{hermitian-scale-bound}
For all positive integers \( d \) and \( r \), there are constants \( \newC{hermitian-scale-bound-multiplier} \) and \( \newC{hermitian-scale-bound-exponent} \) depending on \( d \) and \( r \) such that:

For every number field \( F \) of degree at most \( d \) and every matrix \( A \in \rM_r(F) \), if
\begin{enumerate}[(i)]
\item the entries of \( A \) have denominator at most \( D \),
\item \( \Nm_{F/\bQ} \det A \leq D \), and
\item for every \( \sigma \colon F \to \bC \), \( \sigma(A) \) is Hermitian and positive definite,
\end{enumerate}
then there exists \( Q \in \gGL_r(F) \) such that the entries of \( Q \) are algebraic integers and the entries \( x_{ij} \) of \( Q^\dag A Q \) satisfy
\[ \abs{\sigma(x_{ij})} \leq \refC{hermitian-scale-bound-multiplier} (D \abs{\disc F})^{\refC{hermitian-scale-bound-exponent}} \]
for all \( \sigma \colon F \to \bC \).
\end{lemma}

\begin{proposition} \label{polarisation-bound}
There exists a Galois-compatible weakly bounded diagonal basis \( \basis{3} = \{ \bv{3}{\sigma j} \} \) for \( E_x \) such that the complex numbers
\[ \Theta_x(\bv{3}{\sigma j}, \bv{3}{\sigma j'}) \]
are polynomially bounded for all \( \sigma \in \Sigma \), \( 1 \leq j, j' \leq r_\sigma \).
\end{proposition}

\begin{proof}
Since the Hodge structure \( E_x \) has weight \( 0 \), \( \rho x(i) \) has eigenvalues \( \pm 1 \).
By \cref{totally-real-signs}, we can choose an element \( \zeta \in F_x^\times \) which is totally real, such that \( \sigma(\zeta) \) has the same sign as the eigenvalue of \( \rho x(i) \) on \( \bv{2}{\sigma 1} \) for all \( \sigma \in \Sigma \), and whose height is polynomially bounded.

Define
\[ A_{\sigma jj'} = \sigma(\zeta) \, \Psi_\bC(\barbv{2}{\sigma j}, \bv{2}{\sigma j'}) \]
for each \( \sigma \in \Sigma \) and \( 1 \leq j, j' \leq r_\sigma \).

The way we chose the signs of \( \sigma(\zeta) \) implies that
\[ A_{\sigma jj'} = \abs{\sigma(\zeta)} \, \Theta_x(\bv{2}{\sigma j}, \bv{2}{\sigma j'}). \]
Since \( \Theta_x \) is a positive definite Hermitian form, we deduce that the square matrix~\( A_\sigma \) is Hermitian and positive definite for each \( \sigma \in \Sigma \).

Observe that, for every \( \tau \in \Aut(\bC) \) and \( \sigma \in \Sigma \), we have \( \tau\bar\sigma = \overline{\tau\sigma} \) because \( F_x \) is a product of CM fields and \( \bQ \).
Since \( \Psi \) is defined over \( \bQ \), we deduce that
\[ \tau(A_{\sigma jj'}) = A_{(\tau\sigma) jj'} \]
or in other words there are matrices \( A_i \in \rM_{r_i}(F_i) \) such that
\[ A_{\sigma jj'} = \sigma(A_{i jj'}) \]
for all \( 1 \leq i \leq s \), \( \sigma \in \Sigma_i \), \( 1 \leq j, j' \leq r_i \).

Because \( \basis{2} \) has polynomially bounded denominators and \( \zeta \) has polynomially bounded height, the entries of the \( A_i \) have polynomially bounded denominators.

Furthermore
\[ \prod_{i=1}^s \Nm_{F_i/\bQ} \det A_i = \prod_{\sigma \in \Sigma} \det A_\sigma = \left( \prod_{\sigma \in \Sigma} \abs{\sigma(\zeta)} \right) \det \Theta_x(\basis{2}, \basis{2}). \]
(We are using the fact that the eigenspaces of \( F_x \) are orthogonal to each other with respect to \( \Theta_x \), so the matrix \( \Theta_x(\basis{2}, \basis{2}) \) is block diagonal with one block for each \( \sigma \).)
The values \( \abs{\sigma(\zeta)} \) are polynomially bounded because \( \zeta \) has polynomially bounded height.
Using \cref{det-hermitian-bounded},
we conclude that
\[ \prod_{i=1}^s \Nm_{F_i/\bQ} \det A_i \]
is polynomially bounded.

Each value \( (\Nm_{F_i/\bQ} \det A_i)^{-1} \) is polynomially bounded because the denominators of \( A_i \) are polynomially bounded.
Hence
\[ \Nm_{F_i/\bQ} \det A_i \]
is polynomially bounded for each \( i \).

Hence \( A_i \in \rM_{r_i}(F_i) \) satisfies the conditions of \cref{hermitian-scale-bound} (with \( D \) being polynomial in \( \abs{\disc(R)} \)).
So we can choose \( Q_i \in \gGL_{r_i}(F_i) \) as in the conclusion of that lemma.

Let
\[ \bv{3}{\sigma j} = \sum_{j'=1}^{r_\sigma} \sigma(Q_{ij'j}) \bv{2}{\sigma j'} \]
for all \( 1 \leq i \leq s \), \( \sigma \in \Sigma_i \) and \( 1 \leq j \leq r_\sigma \).

Now \( \basis{3} = \{ \bv{3}{\sigma j} \} \) is a diagonal basis because each \( \bv{3}{\sigma j} \) is a linear combination of \( \sigma \)-eigenvectors for \( F_x \).
Clearly it is also Galois-compatible.

Because the archimedean absolute values of entries of \( Q_i^\dag A_i Q_i \) are polynomially bounded,
\[ \Nm_{F_i/\bQ} \det Q_i^\dag A_i Q_i \]
is polynomially bounded.
As remarked above, \( (\Nm_{F_i/\bQ} \det A_i)^{-1} \) is polynomially bounded.
Hence
\[ \abs{\Nm_{F_i/\bQ} \det Q_i}^2 = (\Nm_{F_i/\bQ} \det Q_i^\dag A_i Q_i)(\Nm_{F_i/\bQ} \det A_i)^{-1} \]
is polynomially bounded.
We deduce that the covolume of \( \basis{3} \), namely
\[ \left( \prod_{i=1}^s \Nm_{F_i/\bQ} \det Q_i \right)  \Vol(\cB_\bZ:\basis{2}), \]
is polynomially bounded.

Furthermore \( \basis{3} \) has bounded denominators because \( \basis{2} \) had bounded denominators and the entries of \( Q_i \) are algebraic integers.
Hence \( \basis{3} \) is weakly bounded.

Finally the values
\[ \abs{\sigma(\zeta)} \, \Theta_x(\bv{3}{\sigma j}, \bv{3}{\sigma j'}) \]
are the entries of \( \sigma(Q_i^\dag A_i Q_i) \) and so are polynomially bounded.
Since \( \zeta \) has polynomially bounded height, we conclude that
\[ \Theta_x(\bv{3}{\sigma j}, \bv{3}{\sigma j'}) \]
are polynomially bounded.
\end{proof}

\subsection{A \hgG-admissible diagonal basis whose height is bounded relative to \texorpdfstring{\( \basis{3} \)}{B(3)}} \label{ssec:g-admissible-basis}

In this section we will construct the basis \( \basis{4} \) which is \( \gG \)-admissible, diagonal, and whose coordinates with respect to \( \basis{3} \) have polynomially bounded height.

We begin by considering the coordinates of \( \basis{3} \) with respect to \( \basis{1} \).
Using the fact that \( \Theta_x \) is polynomially bounded on \( \basis{3} \), we show that these coordinates are polynomially bounded, and deduce that the multilinear form \( \Phi \) is polynomially bounded on \( \basis{3} \).
Because \( \basis{3} \) is Galois-compatible and has polynomially bounded denominators, the bound on (the standard absolute value of) the values of \( \Phi \) on~\( \basis{3} \) implies a bound for the heights of these values.

The set of coordinates (with respect to \( \basis{3} \)) of \( \gG \)-admissible diagonal bases forms an affine algebraic set \( V_0 \) defined over \( \QalgR \).
Because \( \gG \) is precisely the stabiliser of \( \Phi \), we can use \( \Phi \) to write down equations for \( V_{0,\Qalg} \), and then describe a suitable \( \QalgRb \)-form of this algebraic set.
The bound for the heights of values of \( \Phi \) on \( \basis{3} \) implies that these equations have coefficients of polynomially bounded height, and the existence of a \( \gG \)-admissible diagonal basis \( \basis{1} \) implies that \( V_0(\bR) \) is non-empty.
We can therefore use \cref{variety-bound} to obtain an element of \( V_0(\QalgR) \) of polynomially bounded height, and thus \( \basis{4} \).

\begin{lemma} \label{coords13-bounded}
For each \( \sigma \in \Sigma \), let \( R_\sigma \) be the matrix in \( \gGL_{r_\sigma}(\bC) \) such that
\[ \bv{3}{\sigma j} = \sum_{j'=1}^{r_\sigma} R_{\sigma j'j} \bv{1}{\sigma j'}. \]
(Such matrices exist because \( \basis{1} \) and \( \basis{3} \) are both diagonal bases for \( E_x \).)

The entries of \( R_\sigma \) are polynomially bounded.
\end{lemma}

\begin{proof}
Consider the \( r_\sigma \times r_\sigma \) matrices
\[ B_{\sigma jj'} = \Theta_x(\bv{1}{\sigma j}, \bv{1}{\sigma j'}) \]
and
\[ B'_{\sigma jj'} = \Theta_x(\bv{3}{\sigma j}, \bv{3}{\sigma j'}). \]
These matrices are Hermitian positive definite, so they can be diagonalised by unitary matrices:
\[ B_\sigma = U_\sigma^\dag D_\sigma U_\sigma \text{ and } B'_\sigma = U^{\prime\dag}_\sigma D'_\sigma U'_\sigma \]
where \( U_\sigma, U'_\sigma \) are unitary matrices and \( D_\sigma, D'_\sigma \) are diagonal with positive real entries.
We can then choose further diagonal matrices \( \Lambda_\sigma \) with positive real entries such that
\[ \Lambda_\sigma^2 D_\sigma = D'_\sigma. \]

Observe that
\[ B_{\sigma jj'} = \Theta_{x_0}(\bv{0}{\sigma j}, \bv{0}{\sigma j}) \]
and hence does not depend on \( x \).
Thus \( D_\sigma \) and \( U_\sigma \) are constant.

The entries of \( U'_\sigma \) are bounded by~\( 1 \) because it is a unitary matrix.
The entries of \( B'_\sigma \) are polynomially bounded by the construction of \( \basis{3} \).
Hence the entries of \( D'_\sigma \), and therefore also \( \Lambda_\sigma \), are polynomially bounded.

Let
\[ R'_\sigma = U_\sigma^\dag \Lambda_\sigma U'_\sigma \in \gGL_{r_\sigma}(\bC). \]
This has polynomially bounded entries and satisfies
\[ B'_\sigma = {R'}_\sigma^\dag B_\sigma R'_\sigma. \]
But by the definition of \( R \),
\[ B'_\sigma = {R}_\sigma^\dag B_\sigma R_\sigma. \]

Hence \( R_\sigma {R'}_\sigma^{-1} \) lies in the compact group of \( B_\sigma \)-unitary matrices.
Because \( B_\sigma \) is constant, there is a uniform bound for entries of the entries of \( B_\sigma \)-unitary matrices.

Combining this bound for \( R_\sigma {R'}_\sigma^{-1} \) with the polynomial bound for \( R'_\sigma \), we deduce that \( R_\sigma \) has polynomially bounded entries.
\end{proof}

\begin{lemma} \label{mult-form-bound}
The values
\[ \Phi(\bv{3}{\sigma_1 j_1}, \dotsc, \bv{3}{\sigma_k j_k}) \]
are algebraic numbers of polynomially bounded height for all
\[ ((\sigma_1, j_1), \dotsc, (\sigma_k, j_k)) \in (\bigcup_i \Sigma_i \times \{ 1, \dotsc, r_i \})^k. \]
\end{lemma}

\begin{proof}
The definition of \( R_\sigma \) implies that we can express the values of \( \Phi \) on \( \basis{3} \) in terms of its values on \( \basis{1} \) as follows:
\[ \Phi(\bv{3}{\sigma_1 j_1}, \dotsc, \bv{3}{\sigma_k j_k})
   = \sum_{j'_1=1}^{r_{\sigma_1}} \dotsb \sum_{j'_k=1}^{r_{\sigma_k}} R_{\sigma_1 j'_1 j_1} \dotsb R_{\sigma_k j'_k j_k} \Phi(\bv{1}{\sigma_1 j'_1}, \dotsc, \bv{1}{\sigma_k j'_k}).
\]
The facts that \( \basis{1} \) is \( \gG \)-admissible and \( \Phi \) is \( \gG \)-invariant imply that the values of \( \Phi \) on \( \basis{1} \) are equal to its values on \( \basis{0} \) and hence constant.
Thus \cref{coords13-bounded} implies that the values
\[ \Phi(\bv{3}{\sigma_1 j_1}, \dotsc, \bv{3}{\sigma_k j_k}) \]
are polynomially bounded.

The facts that \( \basis{3} \) is Galois-compatible and that \( \Phi \) is defined over \( \bQ \) imply that
\[ \tau(\Phi(\bv{3}{\sigma_1 j_1}, \dotsc, \bv{3}{\sigma_k j_k})) = \Phi(\bv{3}{(\tau\sigma_1) j_1}, \dotsc, \bv{3}{(\tau\sigma_k) j_k}) \]
for all \( \tau \in \Aut(\bC) \).
Hence the values of \( \Phi \) on \( \basis{3} \) are algebraic numbers.

Because each \( \Phi(\bv{3}{\sigma_1 j_1}, \dotsc, \bv{3}{\sigma_k j_k}) \) is polynomially bounded and because these values are Galois-compatible, all Galois conjugates of \( \Phi(\bv{3}{\sigma_1 j_1}, \dotsc, \bv{3}{\sigma_k j_k}) \) are polynomially bounded.
Furthermore the denominators of \( \Phi(\bv{3}{\sigma_1 j_1}, \dotsc, \bv{3}{\sigma_k j_k}) \) are polynomially bounded because \( \basis{3} \) is weakly bounded.
The combination of bounded denominators and all Galois conjugates being bounded implies that the heights are polynomially bounded.
\end{proof}

\begin{proposition} \label{b3-to-b4}
There exist matrices \( S_\sigma \in \gGL_{r_i}(\bC) \) whose entries are algebraic numbers of polynomially bounded height and such that the diagonal basis \( \basis{4} = \{ \bv{4}{\sigma j} \} \) defined by
\[ \bv{4}{\sigma j} = \sum_{j'=1}^{r_\sigma} S_{\sigma j'j} \bv{3}{\sigma j'} \]
is \( \gG \)-admissible.
\end{proposition}

\begin{proof}
Let \( V \) denote the \( \Qalg \)-algebraic set consisting of those elements
\[ S:=(S_\sigma)_{\sigma \in \Sigma} \in \prod_{\sigma \in \Sigma} \gGL_{r_\sigma} \]
which satisfy the equations
\[ \sum_{j_1=1}^{r_1} \dotsb \sum_{j_k=1}^{r_k} S_{\sigma_1 j'_1 j_1} \dotsb S_{\sigma_k j'_k j_k} \Phi(\bv{3}{\sigma_1 j'_1}, \dotsc, \bv{3}{\sigma_k j'_k}) = \Phi(\bv{1}{\sigma_1 j_1}, \dotsc, \bv{1}{\sigma_k j_k}) \tag{*} \label{eqn:var-definition} \]
for all tuples
\[ ((\sigma_1, j_1), \dotsc, (\sigma_k, j_k)) \in \bigl( \bigcup_i \Sigma_i \times \{ 1, \dotsc, r_i \} \bigr)^k. \]
Observe that the basis
\[ \basis{3}S = \bigl\{ \sum_{j'=1}^{r_\sigma} S_{\sigma j' j} \bv{3}{\sigma j'} \mid \sigma \in \Sigma, 1 \leq j \leq r_\sigma \bigr\} \]
is always diagonal.
The purpose of the algebraic set \( V \) is that the basis \( \basis{3}S \) is of the form \( g\basis{0} \) for some \( g \in \gG(\bC) \) if and only if \( S \in V(\bC) \).

We will now construct a \( \QalgRb \)-form \( V_0 \) of \( V \) such that \( \basis{3}{S} \) is \( \gG \)-admissible if and only if \( S \in V_0(\bR) \).

Recall that, when choosing \( \basis{0} \), we required that for each \( v \in \basis{0} \), its complex conjugate is also in \( \basis{0} \).
Since \( \basis{1} = g_1 \basis{0} \) for some real linear map \( g_1 \), it follows that \( \basis{1} \) has the same property.
Clearly, if \( F_x \) acts on \( v \) via the character \( \sigma \), then it acts on \( \bar{v} \) via the character \( \bar\sigma \).
Hence for each \( \sigma \in \Sigma \), there is a permutation \( \beta(\sigma, -) \) of \( \{ 1, \dotsc, r_\sigma \} \) defined by
\[ \barbv{1}{\sigma j} = \bv{1}{\bar\sigma \beta(\sigma, j)}. \]
(The subscript on the right hand side of that equation might be hard to read: it is \( \bar\sigma \beta(\sigma, j) \).)
On the other hand, because \( \basis{3} \) is Galois-compatible, we have
\[ \barbv{3}{\sigma j} = \bv{3}{\bar\sigma j}. \]

Define a semilinear involution \( \theta \) of \( \prod_{\sigma \in \Sigma} \gGL_{r_\sigma,\Qalg} \) by 
\[ \theta(S)_{\sigma j' j} = \bar{S}_{\bar\sigma j' \beta(\sigma, j)}. \]
By construction, a tuple of matrices \( S \in \prod_{\Sigma} \gGL_{r_\sigma}(\bC) \) satisfies \( \theta(S) = S \) if and only if the basis \( \basis{3}S \) is permuted by complex conjugation in the same fashion as \( \basis{1} \), or in other words, if and only if the linear map transforming \( \basis{1} \) into this basis is defined over \( \bR \).
Thus \( \basis{3}S \) is \( \gG \)-admissible if and only if \( S \in V(\bC) \) and \( \theta(S) = S \).

Because \( \Phi \) is defined over \( \bQ \) (so \textit{a fortiori} over \( \bR \)), composing one of the equations \eqref{eqn:var-definition} with \( \theta \) transforms it into another of the equations \eqref{eqn:var-definition}.
Hence there is a \( \QalgRb \)-form \( V_0 \) of \( V \) on which the action of complex conjugation is given by \( \theta \).
In particular
\[ V_0(\bR) = \{ S \in V(\bC) \mid \theta(S) = S \} \]
and thus \( \basis{3}S \) is \( \gG \)-admissible if and only if \( S \in V_0(\bR) \).

The \( \QalgRb \)-variety \( V_0 \) is defined by the equations
\[ f + \theta(f), \; i(f - \theta(f)) \]
as \( f \) runs over the equations \eqref{eqn:var-definition}.
The number of equations \eqref{eqn:var-definition} depends only on the combinatorial data \( s, r_1, \dotsc, r_s \) (for which there are finitely many possibilities), and their degrees are fixed.
By \cref{mult-form-bound}, their coefficients are algebraic numbers of polynomially bounded height.
Hence \( V_0 \) is defined by a finite set of polynomials such that the number of polynomials in the set is uniformly bounded, the degree of each of the polynomials is uniformly bounded and their coefficients are elements of \( \QalgR \) of polynomially bounded height.

Furthermore \( V_0(\bR) \) is non-empty because \( R^{-1} \in V_0(\bR) \) (recall that \( R \) was defined in \cref{coords13-bounded} as the tuple of matrices such that \( \basis{3} = \basis{1}R \)).

Thus we can apply \cref{variety-bound} to \( V_0 \), to deduce that \( V_0(\QalgR) \) contains a point \( S \) of polynomially bounded height.
This is the \( S \) whose existence is asserted by the proposition.
\end{proof}

\subsection{Bounding the height of \texorpdfstring{\( \basis{4} \)}{B(4)}} \label{ssec:reduction-theory}

The main work in this section is showing that the coordinates of \( \basis{4} \) (with respect to \( \cB_\bZ \)) are polynomially bounded in absolute value.
We will then (by going to \( \basis{3} \) and back) deduce that these coordinates have polynomially bounded height.

Notation-wise, we shall consider \( E_\bC^\vee \) as a space of row vectors and write \( \lVert \cdot \rVert \) for the standard Hermitian norm on \( E_\bC^\vee \) with respect to the basis dual to \( \cB_\bZ \).
For any matrix \( M \in \gGL_n(\bC) \), we shall write \( \lambda_{\min}(M) \) for the minimum of the absolute values of eigenvalues of \( M \).

Let \( g_0, g_3, g_4 \in \gGL(E_\bC) \) be the linear maps such that
\[ \basis{0} = g_0 \cB_\bZ, \; \basis{3} = g_3 \basis{0} \text{ and } \basis{4} = g_4 \basis{0}. \]
We will interpret these linear maps as matrices in \( \gGL_n(\bC) \) with respect to the basis \( \cB_\bZ \).
Then the coordinates of \( \basis{3} \) and \( \basis{4} \) (with respect to \( \cB_\bZ \)) are given by the columns of \( g_3 g_0 \) and \( g_4 g_0 \) respectively.

Let \( S' \) denote the matrix in \( \gGL(E_\bC) \) which expresses the coordinates of \( \basis{4} \) with respect to \( \basis{3} \).
In other words,
\[ g_4 g_0 = g_3 g_0 S'. \]
By the construction of \( \basis{4} \),
\( S' \) can be obtained by taking the block diagonal matrix with blocks \( S_\sigma \) from \cref{b3-to-b4}, and permuting the rows and columns to match the orderings of the bases \( \basis{3} \) and \( \basis{4} \).

Since \( \basis{4} \) is \( \gG \)-admissible, we have \( g_4 \in \gG(\bR) \).
Since \( \basis{4} \) is diagonal for \( E_x \), we have \( g_4 x_0 = x \) in the symmetric space~\( X \).
Since \( x_4 \) is in the fundamental set \( \cF = C.\fS.x_0 \), we have \( g_4 \in C.\fS \).
We can therefore write
\[ g_4 = \gamma \nu \alpha \kappa \]
with \( \gamma \in C \), \( \nu \in \omega \), \( \alpha \in A_t \) and \( \kappa \in K \).

\begin{lemma} \label{g4-lattice-bounded}
For all \( e^\vee \in E_\bZ^\vee - \{ 0 \} \),
\[ \lVert e^\vee g_4 \rVert^{-1} \]
is polynomially bounded.
\end{lemma}

\begin{proof}
We have
\[ \lVert e^\vee g_4 \rVert = \lVert e^\vee g_3 \, g_0 S' g_0^{-1} \rVert \geq \lambda_{\min}(S' g_0^{-1}) \lVert e^\vee g_3 g_0 \rVert. \]
By \cref{b3-to-b4}, the entries of \( S' \) are algebraic numbers with polynomially bounded height,
while \( g_0 \) is a constant matrix whose entries are algebraic numbers.
Hence the entries of \( S' g_0^{-1} \) are algebraic numbers with polynomially bounded height,
so the same is true of its eigenvalues.
Hence \( \lambda_{\min}(S' g_0^{-1})^{-1} \) is polynomially bounded.

The row vector \( e^\vee g_3 g_0 \) is a \( \bZ \)-linear combination of rows of \( g_3 g_0 \).
Since \( \basis{3} \) is Galois-compatible and weakly bounded, the coordinates of \( e^\vee g_3 g_0 \) are a set of algebraic numbers which are closed under Galois conjugation and have polynomially bounded denominators.
Hence \( \lVert e^\vee g_3 g_0 \rVert^2 \) is a rational number with polynomially bounded denominator.

Furthermore \( \lVert e^\vee g_3 g_0 \rVert^2 \neq 0 \) because \( g_3 g_0 \) is invertible.
A non-zero rational number with polynomially bounded denominator has polynomially bounded inverse, so \( \lVert e^\vee g_3 g_0 \rVert^{-1} \) is polynomially bounded.
\end{proof}

\begin{lemma} \label{row-lengths-bounded}
For \( 1 \leq j \leq n \),
\[ \lVert e_j^\vee \nu \alpha \rVert^{-1} \]
is polynomially bounded.
\end{lemma}

\begin{proof}
By definition, \( \nu \alpha = \gamma^{-1} g_4 \kappa^{-1} \).

Since \( \kappa \) is contained in the constant compact set \( K \), for any \( v^\vee \in E_\bC^\vee \), \( \lVert v^\vee \kappa^{-1} \rVert \) is bounded between constant multiples of \( \lVert v^\vee \rVert \).
Hence it will suffice to prove that \( \lVert e_j^\vee \gamma^{-1} g_4 \rVert^{-1} \) is polynomially bounded.

Since \( \gamma \) lies in the constant finite set \( C \subset \gG(\bQ) \), there is a uniform bound for the denominator of \( \gamma^{-1} \).
Thus \( e_j^\vee \gamma^{-1} \) lies in a constant rational multiple of \( E_\bZ^\vee \), and so \cref{g4-lattice-bounded} implies that \( \lVert e_j^\vee \gamma^{-1} g_4 \rVert^{-1} \) is polynomially bounded.
\end{proof}

\begin{lemma} \label{root-bound}
There exists a positive integer \( \ell \) depending only on \( \gG \) such that:

For any \( \nu \in \gU(\bR)\gM(\bR)^+ \) and \( \alpha \in A_t \),
if \( \nu_{jj'} \neq 0 \), then
\[ \alpha_{jj} \geq \min(1, t^\ell) \, \alpha_{j'j'}. \]
\end{lemma}

\begin{proof}
The essential step in the proof is Lemma~\ref{correction:um-roots} below.

The proof of Lemma~\ref{correction:um-roots} in the published version of this paper is incorrect.
The point where the proof goes wrong is the claim that the Lie algebra of \( \gU(\bR) \) is the vector space spanned by \( \{ u-1 \mid u \in \gU(\bR) \} \).

Using Lemma~\ref{correction:um-roots}, because there are only finitely many possible choices for \( j \) and~\( j' \), there exists a constant \( \ell \) depending only on \( \gG \) (and not on \( j \), \( j' \)) such that \( \varepsilon_j - \varepsilon_{j'} \) can be expressed as a sum of at most \( \ell \) simple \( \bQ \)-roots.
Since \( \alpha \in A_t \), we conclude that
\[ \alpha_{jj} / \alpha_{j'j'} \geq \min(1, t)^{\ell}. \qedhere \]
\end{proof}

In the published version of this paper, the following lemma is not stated as a free-standing lemma but appears implicitly within the proof of \cref{root-bound}.
However the proof given there is incorrect.
Due to the length of the corrected proof, we have broken it out to a separate lemma.

\theoremstyle{theorem}
\newtheorem{auxlemma}{Lemma}
\renewcommand{\theauxlemma}{\thelemma\Alph{auxlemma}}

\begin{auxlemma} \label{correction:um-roots}
For any \( \nu \in \gU(\bR) \gM(\bR)^+ \), if \( \nu_{jj'} \neq 0 \), then \( \varepsilon_j - \varepsilon_{j'} \) is a positive integer combination of positive roots of \( \gG \) with respect to~\( \gS \).

Here \( \varepsilon_j \) denotes the character of \( \gS \) given by \( \varepsilon_j(\alpha) = \alpha_{jj} \).
\end{auxlemma}

\begin{proof}
Let \( \nu = um \) with \( u \in \gU(\bR) \) and \( m \in \gM(\bR)^+ \).
Since \( \nu_{jj'} \neq 0 \), there exists \( i \in \{ 1, \dotsc, n \} \) such that \( u_{ji} m_{ij'} \neq 0 \).

Let \( \gT \) be a maximal \( \Qalg \)-split torus of \( \gG \) which contains \( \gS \).
According to \cite{humphreys:lag} Proposition~28.1, \( \gU \) is directly spanned by the root groups \( \gU_\beta \) as \( \beta \) runs over
the roots in \( \Phi(\gT, \gG_{\Qalg}) \) for which \( \Lie \gU_\beta \subset \Lie \gU \) -- in other words, those \( \beta \in \Phi(\gT, \gG_{\Qalg}) \) for which \( \beta_{|\gS} \) is a positive root in \( \Phi(\gS, \gG) \).
So we can write
\[ u = u_1 u_2 \dotsm u_r \]
where \( u_i \in \gU_{\gT,\beta_i}(\bC) \) with \( \beta_{i|\gS} \) positive.

For each character \( \delta \in X^*(\gT) \) or \( \varepsilon \in X^*(\gS) \), let \( E_{\gT,\delta} \) or \( E_{\gS,\varepsilon} \) respectively denote the corresponding eigenspace in \( E_\bC \).
By \cite{humphreys:lag} Proposition~27.2, since \( u_i \in \gU_{\gT,\beta_i}(\bC) \), we have
\[ u_i.E_{\gT,\delta}  \subset  \sum_{k \in \bZ_{\geq 0}} E_{\gT,\delta + k\beta_i}. \]
Hence
\[ u.E_{\gT,\delta}  \subset  \sum_{k_1 \in \bZ_{\geq 0}} \cdots \sum_{k_r \in \bZ_{\geq 0}} E_{\gT,\delta + k_1 \beta_1 + \dotsb + k_r \beta_r}. \]
Taking the sum over all \( \delta \in X^*(\gT) \) which restrict to a given \( \varepsilon \in X^*(\gS) \), we get
\[ u.E_{\gS,\varepsilon}  \subset  \sum_{k_1 \in \bZ_{\geq 0}} \cdots \sum_{k_r \in \bZ_{\geq 0}} E_{\gS,\varepsilon + k_1 \beta_{1|\gS} + \dotsb + k_r \beta_{r|\gS}}. \]

Since the torus \( \gS \) is diagonal, we can translate this into a statement about the entries of the matrix \( u \).
Since \( u_{ji} \neq 0 \), we conclude that
\[ \varepsilon_j = \varepsilon_i + k_1 \beta_{1|\gS} + \dotsb + k_r \beta_{r|\gS}. \]
for some \( k_1, \dotsc, k_r \in \bZ_{\geq 0} \).
In other words \( \varepsilon_j - \varepsilon_i \) is a positive integer combination of positive roots of \( \gG \) with respect to \( \gS \).

Since \( m \in Z_\gG(\gS) \) and \( m_{ij'} \neq 0 \), we have \( \varepsilon_i = \varepsilon_{j'} \). 
This completes the proof.
\end{proof}

\begin{lemma} \label{diag-lower-bound}
The values \( \alpha_{jj}^{-1} \) are polynomially bounded.
\end{lemma}

\begin{proof}
We have
\[ \lVert e_j^\vee \nu \alpha \rVert^2 = \sum_{j'=1}^n (\nu_{jj'} \alpha_{j'j'})^2. \]
Using \cref{root-bound}, we deduce that
\[ \lVert e_j^\vee \nu \alpha \rVert^2 \leq \sum_{j'=1}^n \nu_{jj'}^2 \min(1, t^\ell)^{-2} \alpha^2_{jj}. \]

Since \( \nu \) is in the fixed compact set \( \omega \), there is a uniform upper bound for \( \abs{\nu_{jj'}} \).

Hence \( \alpha_{jj} \) is bounded below by a uniform constant multiple of \( \lVert e_j^\vee \nu \alpha \rVert \).
Combining with \cref{row-lengths-bounded}, we deduce that \( \alpha_{jj}^{-1} \) is polynomially bounded.
\end{proof}

\begin{proposition} \label{b4-abs-bound}
The coordinates of \( \basis{4} \) with respect to \( \cB_\bZ \) are polynomially bounded in absolute value.
\end{proposition}

\begin{proof}
Since \( S' \) has determinant of polynomially bounded height and \( \basis{3} \) is weakly bounded, \( g_4 \) has polynomially bounded determinant.
Now
\[ \prod_{j=1}^n \alpha_{jj} = \det \alpha = \abs{\det \gamma}^{-1} \abs{\det g_4} \]
because \( \det \nu = \abs{\det \kappa} = 1 \).
Since \( \gamma \) is in the fixed finite set \( C \), \( \prod \alpha_{jj} \) is polynomially bounded.

Combining this bound for \( \prod \alpha_{jj} \) with the bound for all \( \alpha_{jj}^{-1} \) from \cref{diag-lower-bound}, we deduce that each \( \alpha_{jj} \) is polynomially bounded.

Since each of \( \gamma \), \( \nu \) and \( \kappa \) is contained in a fixed finite set, this implies that the entries of \( g_4 = \gamma \nu \alpha \kappa \) are polynomially bounded.

Since the coordinates of \( \basis{4} \) are given by \( g_4 g_0 \), they are also polynomially bounded.
\end{proof}

\begin{proposition} \label{b4-height-bound}
The coordinates of \( \basis{4} \) with respect to \( \cB_\bZ \) are algebraic numbers of polynomially bounded height.
\end{proposition}

\begin{proof}
Since \( S' \) gives the coordinates of \( \basis{4} \) with respect to \( \basis{3} \) and has entries of polynomially bounded height, \cref{b4-abs-bound} implies that the coordinates of \( \basis{3} \) with respect to \( \cB_\bZ \) are polynomially bounded in absolute value.

Since \( \basis{3} \) is Galois-compatible, this tells us that each coordinate of a vector in \( \basis{3} \) in fact has all its Galois conjugates polynomially bounded.
Furthermore its denominator is polynomially bounded, because \( \basis{3} \) is weakly bounded.
Hence the heights of the coordinates of \( \basis{3} \) are polynomially bounded.

Again using the fact that \( S' \) has polynomially bounded height, we conclude that the heights of the coordinates of \( \basis{4} \) are polynomially bounded.
\end{proof}

\begin{proof}[Proof of \cref{height-disc-r-bound}]
\Cref{b4-height-bound} implies that the entries of \( g_4 \) are algebraic numbers of polynomially bounded height.
Since \( x = g_4 x_0 \) and the action of \( G(\bR) \) on \( X \) is semialgebraic over \( \QalgR \), it follows that the coordinates of \( x \in X \) also are algebraic with polynomially bounded height.
\end{proof}

\section{Points of small height on real affine varieties} \label{sec:variety-bound}

Let \( V \) be an affine algebraic set defined over \( \Qalg \).
Suppose that \( V \) can be defined by a fixed number of polynomials in a fixed number of variables and of fixed degrees, such that the coefficients of these polynomials have height at most \( H \).
Then the arithmetic Bézout theorem and Zhang's theorem on the essential minimum imply that \( V \) has a \( \Qalg \)-point whose height is bounded by a polynomial in \( H \).

For our application to pre-special points, we need a variant of this result: we assume that \( V \) is defined over \( \QalgR \) and we need to show that \( V\QalgRb \) contains a point of polynomially bounded height.
(For us, \( \Qalg \) means the set of algebraic numbers in \( \bC \) rather than an abstract algebraic closure of \( \bQ \), so \( \QalgR \) is well-defined.)
To avoid trivial counter-examples, we must also assume that \( V(\bR) \) is non-empty.

\begin{theorem} \label{variety-bound}
For all positive integers \( m \), \( n \), \( D \), there are constants \( \refC{point-multiplier} \) and \( \refC{point-exponent} \) depending on \( m \), \( n \) and \( D \) such that:

For every affine algebraic set \( V \subset \bA^n \) defined over \( \QalgR \) by polynomials \( \seq{f}{m} \in \QalgRb[\seq{X}{n}] \) of degree at most \( D \) and height at most \( H \), if \( V(\bR) \) is non-empty, then \( V\QalgRb \) contains a point of height at most \( \refC{point-multiplier} H^{\refC{point-exponent}} \).
\end{theorem}

We will prove the theorem by an elementary argument, using a quantitative version of the Noether normalisation lemma.
The proof is by induction on \( n \).

Given an algebraic set in \( \bA^n \), the following lemma constructs an algebraic set in \( \bA^{n-1} \) satisfying appropriate bounds, which we can then study inductively.
Part~(i) of the lemma is simply the inductive step of the Noether normalisation lemma, without any height bound.
Parts~(ii) to~(iv) allow us to find a point of small height in \( V \), once we have got such a point in \( \pi\phi(V) \).
Part~(v) ensures that as we carry out the induction, replacing \( V \subset \bA^n \) by \( \pi\phi(V) \subset \bA^{n-1} \), the number, degrees and heights of the defining polynomials remain controlled. 
(Parts~(i) to~(iii) work when \( k \) is any field of characteristic zero.  For (iv) and (v) the field must be contained in \( \Qalg \) so that we have a notion of height.)

\begin{lemma} \label{normalisation-bound}
Let \( V \subset \bA^n \) be an affine algebraic set defined over a field \( k \subset \Qalg \).
Suppose that \( V \) can be defined by \( m \) polynomials with coefficients in \( k \), each of degree at most \( D \) and height at most \( H \).

If \( V \neq \bA^n \), then there exists a linear automorphism \( \phi \) of \( k^n \) such that:
\begin{enumerate}[(i)]
\item \( (\pi\phi)_{|V} \colon V \to \bA^{n-1} \) is a finite morphism (where \( \pi \colon \bA^n \to \bA^{n-1} \) means projection onto the first \( n-1 \) coordinates);
\item the matrix representing \( \phi \) w.r.t.\ the standard basis has entries in \( \bZ \) with absolute values at most~\( D \);
\item there exists a polynomial \( g \in k[X_1, \dotsc, X_{n-1}][X_n] \) which vanishes on \( \phi(V) \) and which is monic of degree at most \( D \) in \( X_n \) (we do not require that \( \phi(V) = \{ x \mid \pi(x) \in \pi\phi(V) \text{ and } g(x) = 0 \} \));
\item the height of \( g \) is bounded by a polynomial in \( H \);
\item the image \( \pi\phi(V) \) can be defined by polynomials in \( k[X_1, \dotsc, X_{n-1}] \) such that the number and degree of these polynomials are bounded by bounds depending on \( m \) and \( D \) only, and their heights are bounded by a polynomial in \( H \).
\end{enumerate}
Here ``bounded by a polynomial in \( H \)'' means bounded by an expression of the form \( \newC{tmp-multiplier} H^{\newC{tmp-exponent}} \) where \( \refC{tmp-multiplier} \) and \( \refC{tmp-exponent} \) are constants which depend on \( n \), \( m \) and \( D \) only.
\end{lemma}

Using \cref{normalisation-bound}, we can directly prove the simplified version of \cref{variety-bound} which only concerns \( \Qalg \)-points.
Indeed, by induction on \( n \) we get a point
\[ x = (x_1, \dotsc, x_{n-1}) \in \pi\phi(V)(\Qalg) \]
of polynomially bounded height.
Parts~(iii) and (iv) of the lemma imply that any point in the fibre
\( \pi^{-1}(x) \cap \phi(V)(\Qalg) \)
also has polynomially bounded height, because its \( x_n \)-coordinate must be a root of \( g(x_1, \dotsc, x_{n-1})(X_n) \).
Then part~(ii) allows us to transform this into a point of polynomially bounded height in \( V \).

This argument is not sufficient to obtain a \( \QalgRb \)-point of small height on~\( V \).
This is because the fibre \( \pi^{-1}(x) \cap \phi(V) \) at our chosen \( x \) is not guaranteed to contain any real points.

In order to deal with this, we will strengthen the inductive hypothesis slightly as follows:

\begin{theorem} \label{variety-bound-conn-cpt}
In the setting of \cref{variety-bound},
each connected component \( Z \) of \( V(\bR) \) (in the real topology) contains a \( \QalgRb \)-point of height at most \( \polybound{point} \).
\end{theorem}

We then consider the boundary of \( \pi\phi(Z) \), viewed as a subset of \( \pi\phi(V)(\bR) \).
If \( \pi\phi(Z) \) is an entire connected component of \( \pi\phi(V)(\bR) \), then by the inductive hypothesis it contains a \( \QalgRb \)-point of polynomially bounded height and we can apply the same argument as above to get a \( \QalgRb \)-point of small height in~\( Z \).

So the interesting case is when \( \pi\phi(Z) \) has non-empty boundary in \( \pi\phi(V)(\bR) \).
At a point \( (x_1, \dotsc, x_{n-1}) \) on the boundary of \( \pi\phi(Z) \), an analytic argument implies that either \( g(x_1, \dotsc, x_{n-1})(X_n) \) has a double root, or \( \pi\phi(V) \) is singular.
In either case we can find an algebraic subset \( V' \subset V \), of smaller dimension and defined by polynomials satisfying the appropriate bounds on number, degree and height, such that \( Z \cap V'(\bR) \) is non-empty, so we can replace \( V \) by \( V' \) and repeat the argument.

If \( g(x_1, \dotsc, x_{n-1})(X_n) \) has a double root at some point of \( \pi\phi(Z) \), then it is straightforward to construct \( V' \).
In order to construct \( V' \) when \( \pi\phi(V) \) has a singular point, we use \cref{singular-bound}, which we prove using the Chow form and an idea of Bombieri, Masser and Zannier.

\subsection{Definitions}

Throughout this section, \( k \) denotes a field of characteristic \( 0 \).

We will denote by \( \pi \) the linear map \( \bA^n \to \bA^{n-1} \) given by projection onto the first \( n-1 \) coordinates.

We identify an algebraic set over \( k \) with its set of \( \kalg \)-points.
When we say that an algebraic set \( V \) is defined by polynomials \( f_1, \dotsc, f_m \), we mean simply that
\[ V(\kalg) = \{ (x_1, \dotsc, x_n) \in \bA^n(\kalg) \mid f_i(x_1, \dotsc, x_n) = 0 \text{ for all } i \}. \]
In particular, we do not require the ideal \( (f_1, \dotsc, f_m) \subset k[X_1, \dotsc, X_m] \) to be a radical ideal.

Symbols \( C_i \) will denote constants; the first time we introduce each constant we will specify the parameters on which it depends.
Usually \( C_i \) depends on degrees but not heights.

\vskip 1em


We define the \defterm{height} of a point \( (x_1, \dotsc, x_n) \in \bA^n(\Qalg) \) to be the multiplicative absolute Weil height of \( [x_1 : \dotsb : x_n : 1] \in \bP^n(\Qalg) \), that is
\[ \rH(x_1, \dotsc, x_n) = \prod_{v} \max(\abs{x_1}_v, \dotsc, \abs{x_n}_v, 1)^{d_v/[K:\bQ]} \]
where \( K \) is any number field containing \( x_1, \dotsc, x_n \), the product is over all places of \( K \) and \( d_v = [K_v:\bQ_p] \) or \( [K_v:\bR] \) as appropriate, with the absolute values normalised so that they agree with the standard absolute value when restricted to \( \bQ_p \) or \( \bR \).

The \defterm{height} of a polynomial in \( \Qalg[X_1, \dotsc, X_n] \) means the multiplicative absolute Weil height of the point in projective space whose homogeneous coordinates are the coefficients of the polynomial.

\begin{lemma} \label{root-height-bound}
If \( f \in \Qalg[X] \) is a monic polynomial in one variable of degree~\( d \) and \( x \) is a root of \( f \), then
\[ \rH(x) \leq d\rH(f). \]
\end{lemma}

\begin{proof}
Let
\[ f(X) = X^d + a_{d-1} X^{d-1} + \dotsb + a_0. \]
and let \( K \) be a number field containing \( a_0, \dotsc, a_{d-1}, x \).

For each absolute value \( \abs{\cdot} \) of \( K \),
\begin{align*}
    \max(\abs{x}, 1)^d
  &    = \max(\abs{a_{d-1} x^{d-1} + \dotsb + a_0}, 1)
\\& \leq d \max(\abs{a_{d-1} x^{d-1}}, \dotsc, \abs{a_0}, 1)
\\& \leq d \max(\abs{a_{d-1}}, \dotsc, \abs{a_0}, 1) \max(\abs{x}, 1)^{d-1}
\end{align*}
so that
\[ \max(\abs{x}, 1) \leq d \max(\abs{a_{d-1}}, \dotsc, \abs{a_0}, 1) \]
and the factor of \( d \) (coming from the triangle inequality) can be replaced by \( 1 \) for non-archimedean absolute values.

Taking the product over all places of \( K \) gives the result.
\end{proof}

\subsection{Bounded normalisation lemma}

We will now prove parts~(i) to~(iv) of \cref{normalisation-bound}.
Our proof is based on the proof of the Noether normalisation lemma in exercise~5.16 of~\cite{atiyah-macdonald}.

Let us remark that the bound in part~(ii) is stronger than most of our bounds: the height of \( \phi \) is bounded only in terms of the degrees of the defining equations for \( V \), independent of the height of these equations.
We do not need this extra strength.
In order to obtain this bound in part~(ii), we will use the following lemma.

\begin{lemma} \label{single-poly-non-zero-small-point}
Let \( f \in k[X_1, \dotsc, X_n] \) be a non-zero polynomial of degree at most~\( D \) (over any field \( k \) of characteristic zero).
There exist integers \( \lambda_1, \dotsc, \lambda_n \), such that \( \abs{\lambda_i} \leq D \) for all \( i \), \( \lambda_i \neq 0 \) for all \( i \) and
\[ f(\lambda_1, \dotsc, \lambda_n) \neq 0. \]
\end{lemma}

\begin{proof}
The proof is by induction on \( n \).
Write
\[ f(X_1, \dotsc, X_n) = \sum_{i=0}^D g_i(X_1, \dotsc, X_{n-1}) X_n^i. \]
The \( g_i \) are polynomials in \( k[\seq{X}{n-1}] \) of degree at most \( D \).

Choose some \( i \) such that \( g_i \neq 0 \).
By induction, there are non-zero integers \( \seq{\lambda}{n-1} \), with absolute value at most \( D \), such that
\[ g_i(\seq{\lambda}{n-1}) \neq 0. \]

Then \( f(\seq{\lambda}{n-1}, X_n) \) is a non-zero polynomial in one variable \( X_n \).
It has degree at most \( D \), so it has \t most \( D \) roots.

Hence the set of \( D + 1 \) values \( \{ -1, 1, 2, \dotsc, D \} \) contains at least one non-root of \( f(\seq{\lambda}{n-1}, X_n) \), as required.
\end{proof}

\begin{proof}[Proof of \cref{normalisation-bound}~(i)--(iv)]
Let \( f = f_1 \).
This is a polynomial in \( k[\seq{X}{n}] \) of degree at most \( D \) which vanishes on \( V \),
and without loss of generality \( f \neq 0 \).

Let \( F \) denote the homogeneous polynomial consisting of the terms of \( f \) with highest total degree.
By \cref{single-poly-non-zero-small-point}, we can find non-zero integers \( \seq{\lambda}{n} \), with absolute value at most \( D \), such that
\[ F(\seq{\lambda}{n}) \neq 0. \]

Let \( \phi \colon \bA^n \to \bA^n \) be the \( k \)-linear map with matrix
\[ \begin{pmatrix}
   \lambda_n &        &           & -\lambda_1
\\           & \ddots &           & \vdots
\\           &        & \lambda_n & -\lambda_{n-1}
\\           &        &           & 1
\end{pmatrix} \]
This is invertible because \( \lambda_n \neq 0 \).
It satisfies (ii) by construction.

Let \( g \in k[\seq{X}{n}] \) be the polynomial
\[ g(\seq{X}{n}) = f(\phi^{-1}(\seq{X}{n})). \]
Observe that \( g \) vanishes on \( \phi(V) \).

Let \( G \) be the homogeneous polynomial consisting of the terms of \( g \) with highest total degree.
Then
\[ G(0, \dotsc, 0, \lambda_n) = F(\phi^{-1}(0, \dotsc, 0, \lambda_n)) = F(\seq{\lambda}{n}) \neq 0. \]
Since \( G \) is homogeneous and \( \lambda_n \neq 0 \), we deduce that
\[ G(0, \dotsc, 0, 1) \neq 0. \]

Let \( d \) denote the total degree of \( G \).
Then \( G(0, \dotsc, 0, 1) \) is the coefficient of the \( X_n^d \) term of \( g \), and all other terms in \( g \) have a lower power of \( X_n \).
Hence after scaling by \(  G(0, \dotsc, 0, 1)^{-1} \), \( g(\seq{X}{n}) \) is monic of degree \( d \leq D \) as a polynomial in the one variable \( X_n \) with coefficients in \( k[\seq{X}{n-1}] \).
Thus (iii) is satisfied.

The coordinate ring \( k[\phi(V)] \) is generated as a \( k[\seq{X}{n-1}] \)-algebra by \( X_n \).
Furthermore, since \( g \) is monic in \( X_n \) and vanishes on \( \phi(V) \),
we deduce that \( X_n \) (viewed as an element of \( k[\phi(V)] \)) is integral over \( k[\seq{X}{n-1}] \).
Hence \( k[\phi(V)] \) is finite as a \( k[\seq{X}{n-1}] \)-algebra.
In other words, the morphism \( \pi_{|\phi(V)} \) is finite.
Since \( \phi \) is invertible, (i) is satisfied.

Since \( f \), \( \phi \) and \( G(0, \dotsc, 0, 1) \) all have height bounded by a polynomial in \( H \), the same is true of \( g \), so (iv) holds.
\end{proof}

\subsection{Bounding the image of normalisation}

We wish to prove part (v) of \cref{normalisation-bound}, i.e.\ that the image \( \pi\phi(V) \) has polynomially bounded height--complexity.
Note first that \( \pi\phi(V) \) is indeed Zariski-closed because \( \pi\phi_{|V} \) is finite.

Take the polynomials \( f_1 \circ \phi^{-1}, \dotsc, f_m \circ \phi^{-1} \) which define \( V' = \phi(V) \), and label them \( g_1, \dotsc, g_r \), \( h_1, \dotsc, h_s \) such that the \( g_i \) all contain \( X_n \) while the \( h_i \) do not contain \( X_n \).
A point \( (\seq{x}{n-1}) \in \bA^{n-1}(\kalg) \) is in \( \pi\phi(V) \) if and only if all the \( h_i \) vanish at \( (\seq{x}{n-1}) \) and all the one-variable polynomials \( g_i(\seq{x}{n-1})(X_n) \) have a common root \( x_n \).
So the problem is to write down algebraic conditions on the coefficients which recognise when a set of one-variable polynomials with varying coefficients has a common root.

Given two one-variable polynomials, we can recognise when they have a common root using the resultant.
More precisely: for each pair of integers \( (d, e) \), there is a universal polynomial \( \Res_{d,e} \in \bZ[a_0, \dotsc, a_d, b_0, \dotsc, b_e] \) such that for all one-variable polynomials \( u \) of degree \( d \) and \( v \) of degree \( e \), \( \Res_{d,e} \) vanishes on the coefficients of \( u \) and \( v \) if and only if \( u \) and \( v \) have a common root in \( \kalg \).

We must take care about what happens when we plug polynomials into \( \Res_{d,e} \) whose degrees are smaller than the expected degrees \( d \) and \( e \).
It turns out that, if both \( \deg u < d \) and \( \deg v < e \), then \( \Res_{d,e}(u, v) \) is always zero.
If \( \deg u = d \) and \( \deg v < e \) (or vice versa), then as usual, \( \Res_{d,e}(u, v) = 0 \) if and only if \( u \) and \( v \) have a common root.
Essentially, in the case  where both \( \deg u < d \) and \( \deg v < e \), the fact that \( \Res_{d,e}(u, v) = 0 \) is telling us that the degree-\( d \) homogenisation of \( u \) and the degree-\( e \) homogenisation of \( v \) have a common root at \( \infty \in \bP^1(\kalg) \).

To determine whether a set of more than two polynomials have a common root, we just have to look at the resultants of sufficiently many linear combinations of them.

\begin{lemma} \label{multi-resultants}
Let \( \seq{g}{r} \in k[X] \) be polynomials in one variable, with degree at most~\( D \).
Let \( d = \deg g_1 \).

The polynomials \( \seq{g}{r} \) share a common root in \( \kalg \) if and only if
\[ \Res_{d, D}(g_1, \, \lambda_2 g_2 + \lambda_3 g_3 + \dotsb + \lambda_{r-1} g_{r-1} + g_r) = 0 \]
for all \( (\lambda_2, \dotsc, \lambda_{r-1}) \in \{ 0, \dotsc, d \}^{r-2} \).
\end{lemma}

\begin{proof}
The ``only if'' direction is clear: if \( g_1, \dotsc, g_r \) all share a common root, then every linear combination of them also vanishes at that point.

For the ``if'' direction, the proof is by induction on \( r \).
For \( r = 2 \), this is just the definition of the resultant.
We know that \( g_1 \) has exactly degree \( d \), so even if \( \deg g_2 < D \) the resultant still tells us whether they have a common root.

For \( r > 2 \):
Suppose that the given resultants are all zero.
By induction, for every \( \lambda \in \{ 0, \dotsc, d \} \), the polynomials
\[ g_1, g_2, \dotsc, g_{r-2}, \, \lambda g_{r-1} + g_r \]
share a common root \( \alpha_\lambda \).
Since each \( \alpha_\lambda \) is a root of \( g_1 \), and \( g_1 \) has at most \( d \) roots, there must be two distinct values \( \lambda, \mu \in \{ 0, \dotsc, d \} \) such that
\[ \alpha_\lambda = \alpha_\mu. \]
Then by construction
\[ \lambda g_{r-1}(\alpha_\lambda) + g_r(\alpha_\lambda) = \mu g_{r-1}(\alpha_\lambda) + g_r(\alpha_\lambda) = 0 \]
which implies that
\[ g_{r-1}(\alpha_\lambda) = g_r(\alpha_\lambda) = 0. \]

Thus \( \alpha_\lambda \) is the required common root of \( g_1, \dotsc, g_r \).
\end{proof}

The following lemma establishes part~(v) of \cref{normalisation-bound}.
For defining equations of \( V' = \phi(V) \), we take \( f_1 \circ \phi^{-1}, \dotsc, f_m \circ \phi^{-1} \) as well as \( g \) from part (iii) of \cref{normalisation-bound}.
Since \( g \) vanishes on \( V' \) by construction, adding it to the set of defining polynomials does not change \( V' \); \( g \) ensures that the condition that one of the polynomials should be monic in \( X_n \) is satisfied.

\begin{lemma} \label{image-bound}
For all positive integers \( m \), \( n \), \( D \), there are constants \( \newC{image-complexity} \), \( \newC{image-multiplier} \) and \( \newC{image-exponent} \) depending on \( m \), \( n \) and \( D \) such that:

For every affine algebraic set \( V' \subset \bA^n \) defined over \( k \subset \Qalg \) by at most \( m \) polynomials of degree at most \( D \) and height at most \( H \), such that one of these polynomials is monic in \( X_n \), the image under projection
\[ \pi(V') \subset \bA^{n-1} \]
is defined by at most \( \refC{image-complexity} \) polynomials over \( k \), each of degree at most \( \refC{image-complexity} \) and height at most \( \polybound{image} \).
\end{lemma}

\begin{proof}
Take the polynomials defining \( V' \), and label them
\[ g_1, \dotsc, g_r, h_1, \dotsc, h_s \]
such that the \( g_i \) all contain \( X_n \) while the \( h_i \) do not contain \( X_n \).
Choose the labelling such that \( g_1 \) is monic in \( X_n \). 

A point \( (\seq{x}{n-1}) \in \bA^{n-1}(\kalg) \) is in \( \pi(V') \) if and only if all the \( h_i \) vanish at \( (\seq{x}{n-1}) \) and the one-variable polynomials \( g_i(\seq{x}{n-1})(X_n) \) have a common root.

Let \( d \) be the degree of \( g_1 \) with respect to \( X_n \).
Because \( g_1 \) is monic in \( X_n \), \( d \) is equal to \( \deg g_1(x_1, \dotsc, x_{n-1})(X_n) \) for every \( (x_1, \dotsc, x_{n-1}) \in \bA^{n-1}(\kalg) \).

Hence we can apply \cref{multi-resultants} at each point of \( \bA^{n-1}(\kalg) \), the polynomials \( g_i(\seq{x}{n-1})(X_n) \) have a common root if and only if
\[ \Res_{d, D}(g_1, \, \lambda_2 g_2 + \lambda_3 g_3 + \dotsb + \lambda_{r-1} g_{r-1} + g_r) = 0 \]
at \( (\seq{x}{n-1}) \) for all \( (\lambda_2, \dotsc, \lambda_{m-1}) \in \{ 0, \dotsc, d \}^{r-2} \).

These resultants give us a set of \( (d+1)^{r-2} \) polynomials in \( k[\seq{X}{n-1}] \) which, together with \( \seq{h}{s} \), define \( \pi(V') \).
Because these polynomials are constructed as universal polynomials in the coefficients of \( g_1, \dotsc, g_r \),
their degrees are bounded by a constant depending on \( n \) and \( D \) and their heights by a polynomial in \( H \).
\end{proof}

We have completed the proof of \cref{normalisation-bound} and, as remarked below the statement of \cref{normalisation-bound}, it is straightforward to deduce a version of \cref{variety-bound} which only concerns \( \Qalg \)-points.

\subsection{Behaviour of real points under projection}

In order to obtain \( \QalgRb \)-points of small height, and not just \( \Qalg \)-points, we look at the image of a single connected component of \( V(\bR) \) under \( \pi\phi \).
As explained in the introduction to the section the difficult case is when this image is strictly contained in a connected component of \( \pi\phi(V)(\bR) \).

The principle we use to deal with this is that, in a smoothly varying family of monic one-variable polynomials such as \( g(\seq{x}{n-1})(X_n) \), real roots cannot disappear spontaneously: whenever a real root disappears, there must be either a repeated root or a singular point in the base.
The condition that the polynomials are monic ensures that a root cannot disappear off to infinity (such as occurs in \( x_1 X_2 - 1 = 0 \) when \( x_1 \to 0 \)).

The set of points of \( \pi\phi(V) \) at which \( g(\seq{x}{n-1})(X_n) \) has a repeated root is itself an algebraic set, which is (usually) of smaller dimension than \( \pi\phi(V) \) itself, defined by adding \( \partial g/\partial X_n \) to the set of polynomials defining \( \pi\phi(V) \).
Thus clearly it is still defined by polynomials of bounded height, and we can apply the induction hypothesis to this set.

The singular locus of \( \pi\phi(V) \) is likewise an algebraic set of smaller dimension.
Controlling the height of polynomials defining the singular locus is more difficult.
Indeed, instead of using the singular locus directly, we use \cref{singular-bound} to obtain an algebraic set which sits in between \( \pi\phi(V) \) and its singular locus, and whose height we can control.

It is convenient to restrict attention to irreducible algebraic sets in the following lemma.
\Cref{chow-form-bound,chow-polys-bound} ensure that we can take irreducible components of \( V \) while retaining control of the heights of defining polynomials.

\begin{proposition} \label{real-bad-sets}
Let \( W \subset \bA^n \) be an irreducible affine algebraic set defined over \( \bR \) such that \( \pi_{|W} \) is finite.
Let \( g \in \bR[\seq{X}{n-1}][X_n] \) be a polynomial which is monic in \( X_n \) and which vanishes on \( W \).

Let \( Z \) be a connected component of \( W(\bR) \) (in the real topology).

Then at least one of the following holds:
\begin{enumerate}[(i)]
\item \( \pi(Z) \) is a connected component of \( \pi(W)(\bR) \);
\item there is a point \( (\seq{x}{n}) \in Z \) such that \( x_n \) is a repeated root of the one-variable polynomial \( g(\seq{x}{n-1})(X_n) \);
\item \( \pi(Z) \) contains a point at which \( \pi(W) \) is singular.
\end{enumerate}
\end{proposition}

\begin{proof}
Assume for contradiction that none of the statements of the lemma hold.

Since \( \pi_{|W} \) is proper and \( Z \) is a closed subset of \( W(\bR) \), \( \pi(Z) \) is a closed subset of \( \pi(W)(\bR) \).
Furthermore \( \pi(Z) \) is connected because \( Z \) is connected.
Because (i) does not hold, \( \pi(Z) \) is not open in \( \pi(W)(\bR) \).
Hence there exists some point \( (\seq{x}{n-1}) \) which is contained in the intersection of \( \pi(Z) \) with the closure of \( \pi(W)(\bR) - \pi(Z) \).

Consider the algebraic set
\[ \Sigma = \{ (\seq{z}{n}) \in \bA^n \mid (\seq{z}{n-1}) \in \pi(W) \text{ and } g(\seq{z}{n}) = 0 \}. \]
\( W \) is an irreducible Zariski closed subset of \( \Sigma \) of the same dimension, so is an irreducible component of \( \Sigma \).

Since (iii) does not hold, \( \pi(W) \) is smooth at \( (\seq{x}{n-1}) \).
Hence there is an open neighbourhood \( U \) of \( (\seq{x}{n-1}) \) in \( \pi(W)(\bC) \) which is isomorphic (as a complex manifold) to a complex ball.

Since \( (\seq{x}{n-1}) \in \pi(Z) \) we can find some \( x_n \in \bR \) such that
\[ (\seq{x}{n}) \in Z. \]
Since (ii) does not hold, \( x_n \) is a simple root of \( g(\seq{x}{n-1})(X_n) \).
Thus the partial derivative \( \partial g/\partial X_n \) is non-zero at \( (\seq{x}{n}) \).
We can therefore apply the implicit function theorem to get a holomorphic function \( \alpha \colon U \to \Sigma(\bC) \) which induces a bijection between \( U \) and some open neighbourhood \( \tilde{U} \) of \( (\seq{x}{n}) \) in \( \Sigma(\bC) \), such that \( \alpha \) is inverse to \( \pi_{|\tilde{U}} \).

It follows that \( \Sigma \) is smooth at \( (\seq{x}{n}) \) and so \( (\seq{x}{n}) \) is contained in only one irreducible component of \( \Sigma \), namely \( W \).
Thus \( \tilde{U} \subset W(\bC) \).

Furthermore, \( \alpha \) is the only continously differentiable function which induces a bijection \( U \to \tilde{U} \) with inverse \( \pi \).
But since \( \Sigma \) and \( (\seq{x}{n}) \) are defined over \( \bR \), the continuously differentiable function \( \bar{\alpha} = c \circ \alpha \circ c \) (where \( c \) denotes complex conjugation) also induces a bijection \( U \to \tilde{U} \) with inverse \( \pi \).
So the uniqueness of \( \alpha \) implies that
\[ \alpha = \bar\alpha. \]

Hence \( \alpha \) maps \( U \cap \pi(W)(\bR) \) into \( W(\bR) \), and indeed induces a homemomorphism
\[ U \cap \pi(W)(\bR) \to \tilde{U} \cap W(\bR). \]

Since \( W(\bR) \) is a real algebraic set, it is locally connected and \( Z \) is open in \( W(\bR) \).
Hence \( \tilde{U} \cap Z \) is open in \( W(\bR) \).
Thanks to the above homeomorphism, \( \pi(\tilde{U} \cap Z) \) is open in \( \pi(W)(\bR) \).

Hence \( \pi(\tilde{U} \cap Z) \) is an open neighbourhood of \( (\seq{x}{n-1}) \) in \( \pi(W)(\bR) \) which is contained in \( \pi(Z) \).
But the existence of such a neighbourhood contradicts the fact that \( (\seq{x}{n-1}) \) is in the closure of \( \pi(W)(\bR) - \pi(Z) \).
\end{proof}

We are now ready to finish the proof of \cref{variety-bound-conn-cpt}, and hence that of \cref{variety-bound}.

\begin{proof}[Proof of \cref{variety-bound-conn-cpt}]
We prove \cref{variety-bound-conn-cpt} by induction on \( n + \dim V \).

Choose a \( \QalgRb \)-irreducible component \( V_1 \) of \( V \) whose real points intersect~\( Z \).

There are two cases, depending on whether \( V_1 \) is geometrically irreducible or not.
First consider the case in which it is not geometrically irreducible.
Then \( V_1 = V_2 \cup \overline{V_2} \) for some geometrically irreducible algebraic set \( V_2 \) defined over \( \Qalg \),
and \( V_1(\bR) \) is contained in the intersection \( V_2 \cap \overline{V_2} \).
By \cref{chow-form-bound,chow-polys-bound}, \( V_2 \) can be defined by polynomials \( P_1, \dotsc, P_t \) with coefficients in \( \Qalg \) whose number and degree are bounded and whose heights are polynomially bounded.
Then \( V_2 \cap \overline{V_2} \) is defined by the real and imaginary parts
\[ \re P_1, \im P_1, \dotsc, \re P_t, \im P_t \in \QalgRb[\seq{X}{n}] \]
which again are bounded in number and degree and polynomially bounded in height.
Since \( \dim(V_2 \cap \overline{V_2}) < \dim V \), we can apply the induction hypothesis to \( V_2 \cap \overline{V_2} \) and a connected component of \( (V_2 \cap \overline{V_2})(\bR) \cap Z \), proving the theorem in this case.

Otherwise \( V_1 \) is geometrically irreducible.
Acording to \cref{chow-form-bound,chow-polys-bound}, \( V_1 \) can be defined by polynomials with coefficients in \( \QalgR \) whose number and degree are bounded and whose heights are polynomially bounded, so we can replace \( V \) by \( V_1 \) and \( Z \) by a connected component of \( V_1(\bR) \cap Z \).

Choose \( \phi \) and \( g \) as in \cref{normalisation-bound} (with \( k = \QalgR \)) and let \( W = \phi(V) \).
We will need to assume that \( \partial g/\partial X_n \) does not vanish identically on \( W \): to achieve this, simply replace \( g \) by \( \partial^r g/\partial X_n^r \) for the largest \( r < \deg g \) such that this vanishes identically on \( W \) (scaled to remain monic).

If we are in case (i) of \cref{real-bad-sets}, then \( \pi\phi(Z) \) is equal to a connected component \( Z' \) of \( \pi(W)(\bR) \).
By part (v) of \cref{normalisation-bound}, \( \pi(W) \) is defined by polynomials over \( \QalgRb \) which are bounded in number and degree and polynomially bounded in height.
Hence by induction \( Z' \) contains a \( \QalgR \)-point \( (\seq{x}{n-1}) \) of polynomially-bounded height.
Choose \( x_n \) such that \( (\seq{x}{n}) \in Z \); then \( x_n \) must also have bounded height because it is a root of \( g(\seq{x}{n-1})(X_n) \) and the theorem is proved for this case.

If we are in case (ii) of \cref{real-bad-sets}, let
\[ W' = \{ x \in W \mid (\partial g/\partial X_n)(x) = 0 \}. \]
Because we have assumed that \( \partial g/\partial X_n \) does not vanish identically on \( W \), and because \( W \) is irreducible, \( W' \) is a subset of \( W \) of smaller dimension.
Clearly \( W' \) is defined by equations of polynomially bounded height.
Because we are in case (ii) of \cref{real-bad-sets}, \( W'(\bR) \cap Z \) is non-empty.
Hence by induction \( W'(\bR) \cap Z \) contains a \( \QalgRb \)-point of small height.

If we are in case (iii) of \cref{real-bad-sets}, choose \( W' \subset W \) as in \cref{singular-bound}.
By construction, \( W' \) is defined by polynomials of bounded height and has smaller dimension than \( W \).
By case (iii) of \cref{real-bad-sets}, \( W'(\bR) \cap Z \) is non-empty.
Hence by induction \( W'(\bR) \cap Z \) contains a \( \QalgRb \)-point of small height.
\end{proof}

\subsection{Height bounds for varieties and the singular locus}

We will use Chow polynomials, together with an idea of Bombieri, Masser and Zannier, to prove that the singular locus of an irreducible affine algebraic variety is contained in an algebraic subset of positive codimension and which is defined by equations of bounded height.
Along the way, we also get for free a bound for the height of polynomials defining an irreducible component of an affine algebraic set, a necessary reduction in the proof of \cref{variety-bound}.

Let \( W \) be an irreducible algebraic variety defined over \( \Qalg \).
We define a certain polynomial associated with \( W \), the \defterm{Chow form} of \( W \), and define the \defterm{height} \( \rH(W) \) of the variety \( W \) to be the height of its Chow form.
(Because we are working with multiplicative heights, \( \rH(W) \) is \( \exp \mathrm{h}(W) \) where \( \mathrm{h}(W) \) is the height defined by Philippon.)
From the Chow form, we can obtain a special set of polynomials defining \( W \), which we call the \defterm{Chow polynomials} of \( W \).

Our first two lemmas show that, given any set of equations defining \( W \), the height of the Chow form is polynomially bounded with respect to the height of these equations, and in the other direction, the heights of the Chow polynomials are polynomially bounded with respect to the height of the Chow form.
Combining these two lemmas gives us a bound for the heights of equations defining an irreducible component of an affine algebraic set.
Note that some care is required when we consider irreducible components and when we allow arbitrary affine algebraic sets because Chow forms are only defined for irreducible varieties, and the arithmetic Bézout theorem only talks about geometrically irreducible components.

The Chow polynomials do not necessarily generate a radical ideal, and so the algebraic set \( W' \) defined by applying the Jacobian criterion to the Chow polynomials may be bigger than \( \Sing W \).
Following an idea of Bombieri, Masser and Zannier, we show that this \( W' \) is strictly contained in \( W \), and this is enough for us to deal with case (iii) of \cref{real-bad-sets}.

We believe that it should be possible to show that \( \Sing W \) itself can be defined by equations of bounded number and degree and polynomially bounded height by using the theory of Gröbner bases.

Philippon and Nesterenko work with projective rather than affine algebraic sets, but we can turn our affine sets into projective ones simply by embedding \( \bA^n \hookrightarrow \bP^n \) and replacing \( W \) by its Zariski closure (i.e.\ homogenising the defining polynomials).

\begin{lemma} \label{chow-form-bound}
For all positive integers \( m \), \( n \), \( D \), there are constants \( \newC{chow-form-degree} \), \( \newC{chow-form-height-multiplier} \) and \( \newC{chow-form-height-exponent} \) depending on \( m \), \( n \) and \( D \) such that:

For every affine algebraic set \( V \subset \bA^n \) defined over \( \Qalg \) by at most~\( m \) polynomials of degree at most ~\( D \) and height at most~\( H \), each irreducible component of~\( V \) has degree at most~\( \refC{chow-form-degree} \) and height at most
\( \polybound{chow-form-height} \).
\end{lemma}

\begin{proof}
Note that \( V \) is equal to the intersection \( V(f_1)\cap\cdots\cap V(f_m) \), where \( V(f_i) \) is the variety defined by
\[ V(f_i)=\{(x_1,\dotsc,x_n)\in\bA^n \mid f_i(x_1,\dotsc,x_n)=0\}. \]
For an irreducible component \( W \) of \( V \) we may choose an irreducible factor \( g_i \) of each \( f_i \) such that \( W \) is contained in \( V(g_1)\cap\cdots\cap V(g_m) \).

By~\cite{H08} \S2, there exists a constant \( \newC{irred-factor-multiplier}(D) \) such that the height of \( V(g_i) \) is at most \( \refC{irred-factor-multiplier} H \). Of course, the degree of \( V(g_i) \) is at most \( D \). Therefore, by the arithmetic B\'ezout theorem (see~\cite{H08} Theorem 3), there exist constants \( \newC{pairwise-intersection-height-multiplier}(D) \) and \( \newC{pairwise-intersection-height-exponent}(D) \) such that the multiplicative height of any irreducible component of the intersection of any two of the \( V(g_i) \) is at most \( \polybound{pairwise-intersection-height} \). Similarly, the standard B\'ezout theorem (see~\cite{RU09} Lemme 3.4) implies that there exists a constant \( \newC{pairwise-intersection-degree}(D) \) which bounds the degree of such a component.

Reiterating this procedure we eventually reach \( W \) itself (after at most \( m \) steps).
\end{proof}

\begin{lemma} \label{chow-polys-bound}
For all positive integers \( n \) and \( D \), there are constants \( \newC{chow-polys-complexity} \), \( \newC{chow-polys-height-multiplier} \) and \( \newC{chow-polys-height-exponent} \) depending on \( n \) and \( D \) such that:

For every irreducible affine variety \( W \subset \bA^n \) defined over a field \( k \subset \Qalg \) such that \( W \) has degree at most \( D \) and height at most \( H \), the Chow polynomials of \( W \) are a set of at most \( \refC{chow-polys-complexity} \) polynomials in \( k[X_1, \dotsc, X_n] \) defining \( W \), each of degree at most \( \refC{chow-polys-complexity} \) and height at most \( \polybound{chow-polys-height} \).
\end{lemma}

\begin{proof}
We follow~\cite{N77} \S1.
Let \( A = k[\mathbf{X}] \) be the ring of polynomials over \( k \) in the variables \( X_0, \dotsc, X_n \),
let \( r = \dim(W)+1 \) and let \( U_{i,j} \) for \( i\in\{1,\dotsc,r\} \) and \( j\in\{0,\dotsc,n\} \) be algebraically independent variables over \( A \).
We introduce the linear forms
\[ L_i(\mathbf{X}) = \sum_{j=0}^n U_{i,j}X_j \]
and let \( I \) be the homogeneous prime ideal of \( A \) corresponding to the Zariski closure \( W' \) of \( W \) in \( \bP^n \).
We let \( \chi \) denote the ideal in \( A[\mathbf{U}] = A[U_{1,0},\dotsc,U_{r,n}] \) generated by \( X_0,\dotsc, X_n \) and define the ideals
\begin{align*}
   \tilde{I} & =\bigcup_{k\geq 0}((I,L_1,\dotsc,L_r):\chi^k) \subset A[\mathbf{U}],
\\ \bar{I}   & =\tilde{I}\cap k[\mathbf{U}].
\end{align*}

By~\cite{N77} Lemma 5 (3), \( \bar{I} \) is a principal ideal and we choose a generator \( F\in k[\mathbf{U}] \) for \( \bar{I} \).
We refer to this polynomial as a \defterm{Chow form} of \( W' \) (it is unique up to scalar).
The height of \( W \) is, by definition, \( \rH(F) \). In particular, this does not depend on the choice of \( F \).

We consider the variables 
\[ S^{(i)}_{j,k},\ j,k\in\{0,\dotsc,n\},\ i\in\{1,\dotsc,r\}, \]
algebraically independent over \( A \) except for the skew symmetry 
\[ S^{(i)}_{j,k}+S^{(i)}_{k,j}=0. \]
We define an \( A \)-algebra homomorphism
\[ \theta \colon A[\mathbf{U}] \to A[\mathbf{S}] : =A[S^{(1)}_{0,0},\dotsc,S^{(r)}_{n,n}] \]
by sending \( U_{i,j} \) to \( \sum_{k=0}^n S^{(i)}_{j,k}X_k \).
The \defterm{Chow ideal} \( J \) of \( W' \) is the ideal of \( k[\mathbf{X}] \) generated by the coefficients \( P'_1,\dotsc,P'_m \) of \( \theta(F)\) as a polynomial in the \( S^{(i)}_{jk}\ (j<k) \).
We refer to these coefficients as Chow polynomials of \( W' \) and, by~\cite{N77} Lemma 11, they define \( W' \).
Therefore, if we set \( X_0=1 \), we obtain polynomials \( P_1,\dotsc,P_M \) defining~\( W \) and we refer to these as \defterm{Chow polynomials}.
The number \( M \) of Chow polynomials is clearly bounded by a constant depending only on \( n \) and \( D \).

Consider the diagram 
\[ \Gamma \subset \bP^n\times(\bP^n)^r \xrightarrow{\varphi} (\bP^n)^r, \]
where \( \Gamma \) is the variety defined by the ideal generated by \( I \) and the \( L_i \).
By definition, this is a projective and hence proper morphism.
By the proof of~\cite{N77} Lemma 4, the image \( \varphi(\Gamma) \) is precisely the hypersurface \( \mathcal{H} \) defined by \( F \). Therefore, we have a proper, surjective morphism
\[ \Gamma \subset \bP^n\times(\bP^n)^r \xrightarrow{\varphi} \mathcal{H}, \]
which is clearly generically finite.
By the projection formula, the degree of \( \mathcal{H} \) is at most the degree of \( \Gamma \) and, therefore, so is the degree of the Chow form \( F \).

Note that the degree of the subvariety \( \Gamma_W \) of \( \bP^n\times(\bP^n)^r \) defined by the ideal generated by \( I \) in \( A[\mathbf{U}] \) is precisely the degree of \( W \).
Therefore, we can calculate the degree of \( \Gamma \) by applying the standard B\'ezout theorem to the intersection of \( \Gamma_W \) with the subvarieties defined by the \( L_i \).
We conclude that the degree of \( F \) and, hence, the degrees of the Chow polynomials are bounded by a constant \( \refC{chow-polys-complexity}(n, D) \).
Furthermore, given the explicit nature of $\theta$, there exist constants \( \refC{chow-polys-height-multiplier}(n,D) \) and \( \refC{chow-polys-height-exponent}(n,D) \) such that the heights of the Chow polynomials are at most \( \polybound{chow-polys-height} \).
\end{proof}

\begin{proposition} \label{singular-bound}
For all positive integers \( m \), \( n \), \( D \), there are constants \( \newC{singular-complexity} \), \( \newC{singular-multiplier} \) and \( \newC{singular-exponent} \) depending on \( m \), \( n \) and \( D \) such that:

For every geometrically irreducible affine algebraic set \( W \subset \bA^n \) defined over a field \( k \subset \Qalg \) by at most~\( m \) polynomials of degree at most \( D \) and height at most~\( H \), if \( \Sing W(\Qalg) \neq \emptyset \), then there exists an algebraic set \( W' \subset \bA^n \) satisfying:
\begin{enumerate}[(i)]
\item \( \Sing W \subset W' \subset W \);
\item \( W' \neq W \);
\item \( W' \) can be defined by at most \( \refC{singular-complexity} \) polynomials, each of degree at most \( \refC{singular-complexity} \) and having coefficients in \( k \) of height at most \( \polybound{singular} \).
\end{enumerate}
\end{proposition}

\begin{proof}
We follow the argument of~\cite{BMZ07} pp.~9--10.

The hypothesis that \( W \) is geometrically irreducible allows us to apply \cref{chow-form-bound}.
We can then also apply \cref{chow-polys-bound} to deduce that the number and degree of the Chow polynomials of \( W \) are bounded by constants depending only on \( m \), \( n \) and \( D \), and that the heights of the Chow polynomials are polynomially bounded with respect to \( H \).

Let \( J \) be the Chow ideal of \( W \), which is generated by the Chow polynomials \( P_1, \dotsc, P_M \).
Let \( I \) denote the prime ideal of \( W \), and let \( Q_1, \dotsc, Q_{M'} \) be generators for \( I \).
We shall study the rank of the Jacobian matrix
\[ \Jac(P_1, \dotsc, P_M) = \left(\frac{\partial P_i}{\partial X_j}\right)_{i,j} \in \rM_{M,n}(k[W]) \]
and compare it with the rank of the analogous matrix \( \Jac(Q_1, \dotsc, Q_{M'}) \).

According to~\cite{N77} Lemma~11, \( I \) is an isolated primary component of \( J \) (in fact it is the only isolated primary component).
Hence we can find a minimal primary decomposition
\[ J = I \cap I_2 \cap \dotsb \cap I_t. \]
Because this decomposition is minimal, \( I_2 \cap \dotsb \cap I_t \not\subset I \) so we can choose some \( f \in I_2 \cap \dotsb \cap I_t \) which is not in \( I \).
We thus have \( f \not\in I \) but \( fI \subset J \).

Because \( fI \subset J \subset I \), and because all polynomials in \( I \) vanish on \( W \), a calculation shows that
\[ \rk \Jac(fQ_1, \dotsc, fQ_{M'}) \leq \rk \Jac(P_1, \dotsc, P_M) \leq \rk \Jac(Q_1, \dotsc, Q_{M'}) \]
at each point of \( W \).

A similar calculation shows that
\[ \Jac(fQ_1, \dotsc, fQ_{M'}) = f \Jac(Q_1, \dotsc, Q_{M'})  \text{ in } \rM_{M',n}(k[W]). \]
Since \( f \not\in I \), it is invertible in the function field \( k(W) \).
We conclude that
\[ \rk \Jac(fQ_1, \dotsc, fQ_{M'}) = \rk \Jac(Q_1, \dotsc, Q_{M'}) \]
at the generic point of \( W \).
It is well-known that
\[ \rk \Jac(Q_1, \dotsc, Q_{M'}) = n - \dim(W) \]
at the generic point, and we deduce that the same is true of \( \rk \Jac(P_1, \dotsc, P_M) \).

It follows that the inequality
\[ \rk \Jac(P_1, \dotsc, P_M) < n - \dim(W) \]
defines a proper subset \( W' \) of \( W \).
This subset can be defined as the subset on which all \( (n-\dim(W)) \times (n-\dim(W)) \) minors of \( \Jac(P_1, \dotsc, P_M) \) vanish, and so it is an algebraic subset defined by polynomials satisfying the bounds of condition~(iii).

On the other hand, at any point in \( \Sing W \), the Jacobian criterion for singular points tells us that
\[ \rk \Jac(P_1, \dotsc, P_M) \leq \rk \Jac(Q_1, \dotsc, Q_{M'}) < n - \dim(W) \]
and so \( \Sing W \subset W' \).
\end{proof}

\section{Discriminants and tori} \label{sec:discriminants}

The bound for the height of a pre-special point \( x \) in \cref{height-disc-r-bound} is relative to the discriminant of the centre of the endomorphism ring of the Hodge structure \( E_x \), while the bound in \cref{main-bound} (as required to obtain \cref{andre-oort-conditional}) is relative to certain invariants of the Mumford--Tate torus of \( x \).
In this section, we prove the following theorem which links these bounds together.

\begin{theorem}\label{mainindex}
Let \( (\gG, X) \) be a Shimura datum, \( \rho \colon \gG \to \gGL(E) \) a faithful \( \bQ \)-representation and \( E_\bZ \) a lattice in \( E_\bQ \).
Let \( K \) be a compact open subgroup of~\( \gG(\bA_f) \).

There exist constants \( \newC{index-torus-exponent} \), \( \newC{index-disc-exponent} > 0 \) such that for all \( \newC{index-badprime-base} > 0 \), there exists \( \newC{index-multiplier} > 0 \) (where \( \refC{index-torus-exponent} \) and \( \refC{index-disc-exponent} \) depend only on \( \gG \), \( X \), \( \cF \) and the realisation of \( X \), while \( \refC{index-multiplier} \) depends on these data and also \( \refC{index-badprime-base} \)) such that:

For every pre-special point \( x \in X \):

\begin{enumerate}[(a)]
\item Let \( \gM \) denote the Mumford--Tate group of \( x \) (which is a torus because \( x \) is pre-special).
\item Let \( K_{\gM} = K \cap \gM(\bA_f) \) and let \( K_{\gM}^m \) be the maximal compact subgroup of \( \gM(\bA_f) \).
\item Let \( i(\gM) \) be the number of primes \( p \) for which \( K_{\gM,p}:=K \cap \gM(\bQ_p) \) is strictly contained in the maximal compact subgroup \( K^m_{\gM,p} \) of \( \gM(\bQ_p) \).
\item Let \( L \) be the splitting field of \( \gM \).
\item Let \( F_x \) be the centre of \( \End_{\bQ\text{-HS}}E_x \), let \( R_x \) the centre of \( \End_{\bZ\text{-HS}}E_x \), and let \( \cO_x \) be the maximal order in \( F_x \).
\end{enumerate}

Then
\[ [\cO_x:R_x]  \leq  \refC{index-multiplier} \, \refC{index-badprime-base}^{i(\gM)} [K^m_{\gM}:K_{\gM}]^{\refC{index-torus-exponent}} \abs{\disc L}^{\refC{index-disc-exponent}}. \]
\end{theorem}

For an explanation of the quantifiers associated with the constants in the above theorem, see section~\ref{sec:introduction}.
Throughout this section, constants which depend only on \( \gG \), \( X \), \( \rho \) and \( K \) will be referred to as \defterm{uniform}.

Our proof is based on Tsimerman's proof of this theorem for the special case of \( \Ag \) (\cite{T12} Lemma 7.2).
Our only changes to Tsimerman's argument are that in section~\ref{ssec:smallindex}, we consider general CM Hodge structures instead of CM Hodge structures of type \( (-1,0)+(0,-1) \) (which is only a minor change), and that we give a proof for \cref{useUY} which is used tacitly by Tsimerman but which appears to us to be a non-trivial statement.

We will reduce to the case in which the compact open subgroup \( K \subset \gG(\bA_f) \) is a direct product of compact open subgroups \( K_p \subset \gG(\bQ_p) \), so that
\[ [K^m_{\gM}:K_{\gM}] = \prod_p [K^m_{\gM,p}:K_{\gM,p}]. \]

We then study the \( p \)-part of \( [\cO_x:R_x] \), one prime at a time.
We use the fact that tori splitting over unramified local fields have nice integral models to show that, outside a fixed finite set of bad primes, if \( p \) divides \( [\cO_x:R_x] \) then either \( p \) is ramified in \( L \) (i.e.\ \( p \) divides \( \disc L \)) or \( K^m_{\gM,p} \neq K_{\gM,p} \).

Note that, even if \( K^m_{\gM,p} \neq K_{\gM,p} \), it need not be true that \( p \) divides \( [K^m_{\gM,p} : K_{\gM,p}] \).
We will deal with this by using \cite{uy:andre-oort} Proposition~3.15 which says that, outside the bad primes and those ramified in $L$, if \( K^m_{\gM,p} \neq K_{\gM,p} \) then \( [K^m_{\gM,p} : K_{\gM,p}] \gg p \).

The hardest parts of the proof (section~\ref{ssec:smallindex} and \cref{indexexp}) control which power of \( p \) divides \( [\cO_x:R_x] \), relative to the power of \( p \) which divides \( [K^m_{\gM,p} : K_{\gM,p}] \).

\subsection{Generating \texorpdfstring{$\bQ$}{Q}-algebras}

Let $x \in X$ be a pre-special point, $\gM$ its Mumford--Tate group and $L$ the splitting field of $\gM$.
We can naturally embed $\gM(\bQ)$ as a subgroup of $F_x^\times$.
Hence the reciprocity morphism
\[ r \colon \Res_{L/\bQ} \bG_m \to \gM \]
induces a homomorphism
\[ L^\times = \Res_{L/\bQ} \bG_m(\bQ) \to F_x^\times. \]
We will use Goursat's lemma for algebras to show that the image of \( L^\times \) generates \( F_x^\times \) as a \( \bQ \)-algebra.

We will reuse the notation from the proof of \cref{mt-to-diagonal-basis}:
the isotypic decomposition of the \( \bQ \)-Hodge structure \( E \) is
\begin{align*}
E_x = \bigoplus_{i=1}^s E_{x,i}
\end{align*}
where $E_{x,i}$ is isomorphic to a power of an irreducible $\bQ$-Hodge structure with endomorphism algebra $F_i$, the centre of $\End_{\bQ\text{-HS}} E_x$ is
\begin{align*}
F_x = \prod_{i=1}^s F_i,
\end{align*}
and the centre of the centraliser of the Mumford--Tate group $\gM$ in $\gGL(E)$ is denoted~$\gZ$.
We recall that $\gZ(\bQ)$ is the image of $F_x^\times \hookrightarrow \gGL(E_\bQ)$, and so
\begin{align*}
\gZ \cong \prod_{i=1}^s \Res_{F_i/\bQ} \bG_m,
\end{align*}
and that $\gM \subset \gZ$ because $\gM$ is commutative.
Hence we get an injective homomorphism $\gM(\bQ) \to F_x^\times$.

Now the image of $\gM(\bQ)$ generates $F_x$ as a $\bQ$-algebra.
This is an instance of the general fact that for any semisimple representation of a commutative group, the group generates the centre of its endomorphism algebra, but we will give more details of the proof below.

\begin{lemma}\label{factor}
Suppose that the image of $\gM$ under the projection
\begin{align*}
\gZ \to \Res_{F_i/\bQ} \bG_m
\end{align*}
is contained in a subtorus $\Res_{F'/\bQ}\bG_m \subset \Res_{F_i/\bQ}\bG_m$ for some subfield $F'$ of $F_i$. Then $F'=F_i$.
\end{lemma}

\begin{proof}
Consider an irreducible $\gM$-subrepresentation $E' \subset E_{x,i}$.
This is an $F_i$-vector space of dimension~$1$, and hence an $F'$-vector space of dimension~$[F_i:F']$.
Since $\gM(\bQ) \subset F'$, every $F'$-subspace of $E'$ is an $\gM$-subrepresentation.
But since $E'$ is irreducible, we must have $[F_i:F'] = 1$.
\end{proof}

\begin{lemma}\label{factors}
Suppose the map from $\gM$ to $\Res_{F_i/\bQ}\bG_m\times \Res_{F_j/\bQ}\bG_m$ factors through the map
\begin{align*}
\Res_{F_i/\bQ}\bG_m\rightarrow\Res_{F_i/\bQ}\bG_m\times \Res_{F_j/\bQ}\bG_m:x\mapsto(x,\varphi(x)),
\end{align*}
induced by a field isomorphism $\varphi:F_i \rightarrow F_j$. Then $i = j$.
\end{lemma}

\begin{proof}
Let $E_{x,i}'$ and $E_{x,j}'$ denote irreducible $\gM$-subrepresentations of $E_{x,i}$ and $E_{x,j}$ respectively.
The field isomorphism $\varphi$ induces an isomorphism of $\gM$-representations $E'_{x,i} \to E'_{x,j}$ and hence these are both in the same isotypic component i.e.\ $i=j$.
\end{proof}

These properties yield the following conclusion:

\begin{proposition}
The image of $\gM(\bQ)$ generates the $\bQ$-algebra $F_x = \prod_i F_i$.
\end{proposition}

\begin{proof}
Consider the image of $\gM(\bQ)$ in $\Res_{F_i/\bQ}\bG_m(\bQ)$ for each $i$. The $\bQ$-algebra generated by these elements constitutes a subfield $F'$ of $F_i$. However, by Lemma~\ref{factor}, $F'=F_i$. Therefore, the image of $\gM(\bQ)$ in $F_x$ generates a $\bQ$-algebra that surjects on to each factor $F_i$. Furthermore, however, by Lemma~\ref{factors} the projection of $\gM(\bQ)$ to $F_i \times F_j$ for any distinct $i,j$ is not contained in the graph of an isomorphism $F_i \cong F_j$. Therefore, the proposition follows from the following lemma. 
\end{proof}

\begin{lemma}\label{goursat}
Let $A=\bigoplus_{i=1}^r K_i$ be a direct sum of finite field extensions $K_i$ of a field $k$. Let $\Theta$ be a subset of $A$. Then $\Theta$ generates $A$ as a $k$-algebra if and only if
\begin{enumerate}
\item[(1)] for each projection $\pi_i:A\rightarrow K_i$, the image $\pi_i(\Theta)$ is not contained in a proper subfield of $K_i$ and
\item[(2)] for each projection $\pi_{i,j}:A\rightarrow K_i\oplus K_j$, the image $\pi_{i,j}(\Theta)$ is not contained in the graph of a $k$-isomorphism $K_i\cong K_j$.
\end{enumerate}
\end{lemma}

\begin{proof}
This is an immediate consequence of Goursat's lemma for algebras.
\end{proof}

We will now consider the reciprocity morphism $r\colon\Res_{L/\bQ}\bG_m\rightarrow\gM$.
Let us briefly recall its definition.
Extending $x \colon \bS \to \gG_\bR$ to a morphism over $\bC$ yields
\begin{align*}
\bG^2_{m,\bC}\cong\bS_{\bC}\xrightarrow{h_{\bC}}\gM_{\bC}
\end{align*}
and we denote by $\mu$ the cocharacter obtained by pre-composing with the embedding $\bG_{m,\bC}\hookrightarrow\bG^2_{m,\bC}$ sending $z$ to $(z,1)$. As a cocharacter of $\gM$, $\mu$ is necessarily defined over $L$ i.e.\ $\mu:\bG_{m,L}\rightarrow\gM_L$ and restricting scalars to $\bQ$ yields $\Res_{L/\bQ}\bG_m\rightarrow\Res_{L/\bQ}\gM_L$. The morphism $r$ is obtained by composing with the norm map $\Res_{L/\bQ}\gM_L\rightarrow\gM$. Since $r$ is surjective (as a morphism of algebraic groups) we deduce the following corollary:

\begin{corollary}\label{reciprocity}
The image of $\Res_{L/\bQ}\bG_m(\bQ)$ generates the $\bQ$-algebra $F_x$. 
\end{corollary}

\begin{proof}
If the corollary did not hold then the algebra generated by the image of $\Res_{L/\bQ}\bG_m(\bQ)$ in $F_x$ would violate one of the conditions of Lemma~\ref{goursat}. This would imply that the image of $\Res_{L/\bQ}\bG_m$ in the product of the $\Res_{F_i/\bQ}\bG_m$ would factor non-trivially in one of the two ways referred to in Lemmas~\ref{factor} and~\ref{factors}. However, since $r$ is surjective (as a morphism of algebraic groups) and the $\bQ$-points of $\Res_{L/\bQ}\bG_m(\bQ)$ are Zariski dense in $\Res_{L/\bQ}\bG_m$, this would be a contradiction.   
\end{proof}

\subsection{Generators of small index}\label{ssec:smallindex}

We have seen that the reciprocity map induces a homomorphism \( L^\times \to F_x^\times \) whose image generates \( F_x \) as a \( \bQ \)-algebra.
We will construct a \( \bQ \)-linear map \( S \colon L \to F_x \) which can be interpreted as the logarithm of this homomorphism (this interpretation can be made precise \( p \)-adically).
We then show that, for each prime \( p \), there is an element \( z_p \in \cO_L \otimes \bZ_p \) such that \( S(z_p) \) generates a \( \bZ_p \)-algebra whose index in the maximal order of \( F_x \otimes \bQ_p \) is a bounded power of \( p \).
We will later use \( S(z_p) \) to bound the power of \( p \) dividing \( [\cO_x:R_x] \) relative to the power of \( p \) dividing \( [K^m_{\gM,p}:K_{\gM,p}] \).

The morphism $L^{\times}\rightarrow\gM(\bQ)$ defined by $r$ is given explicitly by
\begin{align*}
z\mapsto\prod_{\sigma\in\Gal(L/\bQ)}\mu^{\sigma}(\sigma(z))=\prod_{\sigma\in\Gal(L/\bQ)}\sigma(\mu(z)).
\end{align*}
For $i=1,\dotsc,s$, let $\mu_i$ denote the composition of $\mu \colon \bG_{m,L} \to \gM_L$ with the projection to $(\Res_{F_i/\bQ} \bG_m) \times_\bQ L$.
Since $\gM$ is split over $L$, the fields $F_i$ are all isomorphic to subfields of the Galois field $L$ and so
\begin{align*}
(\Res_{F_i/\bQ} \bG_m) \times_\bQ L \cong \prod_{\tau \colon F_i \to L} \bG_{m,L}.
\end{align*}

Let $(-n_{i,\tau}, -n'_{i,\tau})$ denote the Hodge type of the $F_i$-eigenspace in $E_{x,L}$ on which $F_i$ acts via $\tau$.
Then the definition of the cocharacter $\mu$ of $\gM$ implies that
\begin{align*}
\mu_i(z) = (z^{n_{i,\tau}})_{\tau \colon F_i \to L} \text{ in } \prod_{\tau \colon F_i \to L} \bG_{m,L}.
\end{align*}

Therefore, the projection $r_i$ of $r$ to $\Res_{F_i/\bQ}\bG_m$ sends $z\in L^{\times}$ to the element 
\begin{align*}
\prod_{\sigma\in\Gal(L/\bQ)}\sigma(\mu_i(z))=
\left( \prod_{\sigma\in\Gal(L/\bQ)} \sigma(z)^{n_{i, \sigma^{-1}\tau}} \right)_{\tau \colon F_i \to L}
\end{align*}
in $\Res_{F_i/\bQ} \bG_m(L) \cong \prod_{\tau \colon F_i \to L} L^\times$,
for uniform integers $n_{i,\tau}$.
Observe that this value is $\Gal(L/\bQ)$-invariant, so in fact it is an element of $\Res_{F_i/\bQ} \bG_m(\bQ) = F_i^\times$.

We define a $\bQ$-linear map $S \colon L \to F_x$ by
\begin{align*}
S(z) \otimes 1 =
\left( \sum_{\sigma\in\Gal(L/\bQ)} n_{i,\sigma^{-1}\tau} \, \sigma(z) \right)_{i, \tau}
\text{ in }
F_x \otimes_\bQ L \cong \prod_{i=1}^s \prod_{\tau \colon F_i \to L} L^\times.
\end{align*}
Observe that the right hand side of the above expression is invariant under the action of $\Gal(L/\bQ)$ on $F_x \otimes_\bQ L$, so it does indeed give us values in $F_x \otimes 1$.
This map $S$ is the same as the map from the tangent space of $\Res_{L/\bQ} \bG_m$ to the tangent space of $\gZ$ induced by $r$.

\begin{lemma}\label{Sgen}
The image of $S$ generates the full $\bQ$-algebra $F_x$.
\end{lemma}

\begin{proof}
By \cref{reciprocity}, we know that $r_i \colon \Res_{L/\bQ} \bG_m \to \Res_{F_i/\bQ} \bG_m$ does not factor through $\Res_{F'/\bQ} \bG_m$ for any proper subfield $F' \subset F_i$.
Since $r_i$ is a morphism of algebraic groups over a field of characteristic zero, it follows that the induced map of tangent spaces does not factor through the tangent space of the algebraic subgroup $\Res_{F'/\bQ} \bG_m$.
In other words, the projection of $S$ onto $F_i$ does not factor through any proper subfield $F' \subset F_i$.

A similar argument shows that the projection of $S$ onto $F_i \times F_j$ does not factor through the graph of any isomorphism $F_i \to F_j$ (when $i \neq j$).

By Lemma \ref{goursat}, these properties imply the statement of the lemma.
\end{proof} 

For each prime~$p$, the elements of $\Gal(L/\bQ)$ extend $\bQ_p$-linearly to the $\bQ_p$-algebra $L\otimes_\bQ\bQ_p$ and so the map $S$ also extends $\bQ_p$-linearly to  $L\otimes\bQ_p$. 
Provided $n>1$ (to ensure convergence), we have 
\begin{align*}
r(\exp p^nz)= \exp p^nS(z)
\end{align*}
as maps $\cO_L \otimes \bZ_p \to (\cO_x \otimes \bZ_p)^\times$.

\begin{proposition}\label{elementindex}
There is a uniform constant $\newC{elementindex-exponent}$ such that, for every prime $p$, there exists $z\in\cO_L\otimes\bZ_p$ such that the index of $\bZ_p[S(z)]$ in 
\begin{align*}
\cO_{x,p}:=\cO_x \otimes_\bZ \bZ_p = \prod_{i=1}^s \cO_{F_i} \otimes_\bZ \bZ_p
\end{align*}
is at most $p^{\refC{elementindex-exponent}}$.
\end{proposition}

\begin{proof}
Enumerate the complete set of ring homomorphisms $\psi_i$ from $\cO_{x,p}$ to a fixed $p$-adic completion $L_{v}$ of $L$. Let $v_p$ be the $p$-adic valuation on $\bQ$ and its extension to $L_v$. For $z\in\cO_L\otimes\bZ_p$ the index of $\bZ_p[S(z)]$ in $\cO_{x,p}$ is given by $p$ raised to the value of $v_p$ on the Vandermonde determinant
\begin{align*}
\iota(z):=\prod_{i\neq j}(\psi_i(S(z))-\psi_j(S(z))),
\end{align*}
divided by the discriminant of $\cO_{x,p}$.

To find a suitable $z$ we first choose a $\bZ_p$-basis $\{\beta_j\}$ of $\cO_L\otimes\bZ_p$ such that $v_p(\beta_j)$ is bounded by a uniform positive integer. To see that this is possible, first choose any $\bZ_p$-basis $\{\beta_j\}$ of $\cO_L\otimes\bZ_p$ and recall that, if we enumerate the automorphisms $\sigma_i$ of $L\otimes\bQ_p$, then the discriminant of $L\otimes\bQ_p$ is $p$ raised to the value of $v_p$ on the determinant of the matrix $(\sigma_i(\beta_j))$. This discriminant is uniformly bounded and so we conclude that $v_p(\beta_j)$ is uniformly bounded for one of the $\beta_j$. However, by the properties of $v_p$ we can therefore obtain our desired basis by, if necessary, taking sums of the form $\beta_j+\beta_{j'}$. Since the value of $v_p$ is preserved under automorphism, $v_p(\iota(\beta_j))$ is uniformly bounded for any $\beta_j$.

Writing
\[ z = \sum_j a_j \beta_j \text{ (for } a_j \in \bZ_p \text{)}, \]
we can express $\iota$ as a polynomial over $\bZ_p$ of uniformly bounded degree in the $a_j$. By Lemma \ref{Sgen}, this polynomial is not identically zero.
Hence by Lemma \ref{single-poly-non-zero-small-point}, we can find a linear combination $z$ of the $\beta_j$ whose coefficients have uniformly bounded $p$-adic valuation such that $\iota(z)\neq 0$.
\end{proof}

\subsection{Generating the maximal order}

In this section we show that, for all primes $p$ outside of a uniformly fixed finite set $\Sigma$, if $p$ divides $[\cO_x:R_x]$ then either $p$ is ramified in $L$ (and hence divides $\disc L$) or $K_{\gM,p}$ is strictly contained in $K_{\gM,p}^m$.
Note that, once we know that $K_{\gM,p} \neq K_{\gM,p}^m$, \cite{uy:andre-oort} Proposition~3.15 implies that
\[ [K^m_{\gM,p}:K_{\gM,p}] \gg p. \]

The uniform finite set which we have to omit is the following: let $\Sigma$ be the set of primes $p$ for which the compact open subgroup $K_p \subset \gG(\bQ_p)$ is not contained in $\gGL(E_{\bZ_p})$.

\begin{lemma}\label{generates}
If $p$ is unramified in $L$, the group $K^m_{\gM,p}$ generates the ring $\cO_{x,p}$.
\end{lemma}

\begin{proof}

Recall that $\Res_{L/\bQ}\bG_m$, $\gM$ and $\Res_{F_i/\bQ}\bG_m$ possess smooth models $\mathcal{R}_L$, $\mathcal{T}$ and $\mathcal{R}_{F_i}$ over $\bZ_p$ whose generic fibres are $(\Res_{L/\bQ}\bG_m)_{\bQ_p}$, $\gM_{\bQ_p}$ and $(\Res_{F_i/\bQ}\bG_m)_{\bQ_p}$, respectively. Furthermore, since $p$ is unramified in $L$, we may insist that the special fibres of these models are tori over $\bF_p$ and that their $\bZ_p$-points are the respective unique hyperspecial (maximal) compact open subgroups. The extension of $r$ over $\bQ_p$ then extends uniquely to a morphism
\begin{align*}
\mathcal{R}_L\rightarrow\mathcal{T}\rightarrow\prod_{i=1}^s\mathcal{R}_{F_i}
\end{align*}
over $\bZ_p$ and, reducing modulo $p$, we obtain a map
\begin{align*}
\mathcal{R}_{L,\bF_p}\rightarrow\mathcal{T}_{\bF_p}\rightarrow \prod_{i=1}^s\mathcal{R}_{F_i,\bF_p}.
\end{align*}

Since $\gM(\bQ)$ generates $F_x$, $\gM(\bQ_p)$ generates 
\begin{align*}
F_x\otimes\bQ_p=\prod_{i=1}^s\prod_{v|p}F_{i,v}
\end{align*}
where the second product is over places of $F_i$ dividing $p$.

Therefore, by Lemma \ref{goursat}, the embedding of $\gM_{\bQ_p}$ in 
\begin{align*}
\prod_{i=1}^s(\Res_{F_i/\bQ}\bG_m)_{\bQ_p}=\prod_{i=1}^s\prod_{v|p}\Res_{F_{i,v}/\bQ_p}\bG_m
\end{align*}
has the analogous properties to the embedding over $\bQ$, namely that the projection to any factor $\Res_{F_{i,v}/\bQ_p}\bG_m$ does not factor through $\Res_{F'/\bQ_p}\bG_m$ for any proper $p$-adic subfield $F'$ of $F_{i,v}$ and the projection to the product of any two factors $\Res_{F_{i,v}/\bQ_p}\bG_m$ and $\Res_{F_{j,w}/\bQ_p}\bG_m$ does not factor through the graph of an isomorphism. By~\cite{DG63} Exp. X Lemme 4.1, these two properties immediately transfer to the analogous conditions on the embedding of $\mathcal{T}_{\bF_p}$ in
\begin{align*}
\prod_{i=1}^s\mathcal{R}_{F_i,\bF_p} = \prod_{i=1}^s\prod_{v|p}\Res_{k_{i,v}/\bF_p}\bG_m,
\end{align*}
where $k_{i,v}$ denotes the residue field of $F_{i,v}$. By Lemma \ref{goursat}, we conclude that $\mathcal{T}_{\bF_p}(\bF_p)$ generates the full $\bF_p$-algebra 
\begin{align*}
\cO_x \otimes_\bZ \bF_p = \prod_{i=1}^s \prod_{v|p} k_{i,v}
\end{align*}
The lemma is now a direct consequence of Nakayama's lemma.
\end{proof}

\begin{lemma}\label{inclusion}
If $p\notin\Sigma$, then $K_{\gM,p}$ is contained in $R_{x,p}$.
\end{lemma}

\begin{proof}
Since $p \notin \Sigma$,
\begin{align*}
K_{\gM,p} \subset \gM(\bQ_p)\cap\gGL(E_{\bZ_p}).
\end{align*}
On the other hand, we have already seen that
\begin{align*}
\gM(\bQ_p) \subset \gZ(\bQ_p) = (F_x \otimes_\bQ \bQ_p)^\times.
\end{align*}
Hence
\[ K_{\gM,p} \subset (F_x \otimes_\bQ \bQ_p) \cap \End_{\bZ_p} E_{\bZ_p} = R_{x,p}.
\qedhere \]
\end{proof}

\begin{proposition}\label{useUY}
If $p\notin\Sigma$ is unramified in $L$ and $p$ divides $[\cO_x:R_x]$ then $K_{\gM,p}$ is strictly contained in $K^m_{\gM,p}$. 
\end{proposition}

\begin{proof}
Suppose $K_{\gM,p}=K^m_{\gM,p}$. By Lemma \ref{generates}, $K^m_{\gM,p}$ generates the ring $\cO_{x,p}$. However, by Lemma \ref{inclusion}, $K^m_{\gM,p}$ is contained in the ring $R_{x,p}$. Hence $R_{x,p}=\cO_{x,p}$.
\end{proof}

\subsection{Proof of Theorem \ref{mainindex}}

The following proposition is the final ingredient in the proof of \cref{mainindex}.
We use it in two different ways: if $p$ divides $[K^m_{\gM,p}:K_{\gM,p}]$, then the proposition tells us that the $p$-part of $[\cO_x:R_x]$ is bounded by a suitable power of the $p$-part of $[K^m_{\gM,p}:K_{\gM,p}]$, while if $p$ does not divide $[K^m_{\gM,p}:K_{\gM,p}]$, then the proposition tells us that the power of $p$ dividing $[\cO_x:R_x]$ is uniformly bounded.

\begin{proposition}\label{indexexp}
There is a uniform integer $\newC{indexexp}$ such that, if $p^n$ is the maximum power of $p$ dividing $[K^m_{\gM,p}:K_{\gM,p}]$, then
\begin{align*}
p^{\refC{indexexp}(n+1)}\nmid[\cO_x:R_x].
\end{align*}
\end{proposition}

\begin{proof}
Choose $z\in\cO_L\otimes\bZ_p$ as in Proposition \ref{elementindex} so that the index of $\bZ_p[S(z)]$ in $\cO_{x,p}$ is at most $p^{\refC{elementindex-exponent}}$. Let $h$ be the dimension of $\cO_{x,p}$ as a free module over $\bZ_p$ and define $y_j:=\exp jp^NS(z)$ for $j\in\{0,...,h-1\}$ and $N>1$ an integer. Note that $y_j=r(\exp jz)^{p^N}$ is an element of $K^m_{\gM,p}$. However, it also lies in the image $I$ of the $p$-adic exponential on $p^2\cO_{x,p}$, which is a pro-$p$ group. Therefore, the quotient of $K^m_{\gM,p}\cap I$ by $K_{\gM,p}\cap I$ is a $p$-group contained in the finite abelian group $K^m_{\gM,p}/K_{\gM,p}$. In particular, its order divides $p^n$ and, if $n$ divides $N$, we conclude that $y_j\in K_{\gM,p}$. If $p\notin\Sigma$ we conclude that $y_j\in R_{x,p}$. If $p\in\Sigma$, we can replace $N$ by a uniformly bounded multiple of $n$ to obtain the same result.

By definition, modulo $p^{Nh-h}$,
\begin{align*}
y_j\equiv 1+jp^NS(z)+\frac{j^2p^{2N}S(z)^2}{2!}+\cdots+\frac{j^{h-1}p^{N(h-1)}S(z)^{h-1}}{(h-1)!}.
\end{align*}
Therefore, the span $\Lambda$ of the elements $y_j$ in $\bZ_p[p^NS(z)]$ modulo $p^{Nh-h}$ is of index bounded by the maximum power $p^{m_h}$ of $p$ dividing the (non-zero) determinant of the matrix 
\begin{align*}
\left(\frac{j^k}{k!}\right)_{j,k\in\{0,...,h-1\}}.
\end{align*}
On the other hand, $\bZ_p[p^NS(z)]$ certainly contains $p^{N(h-1)+\refC{elementindex-exponent}}\cO_{x,p}$, so we conclude that $\Lambda$ contains 
\begin{align*}
p^{N(h-1)+\refC{elementindex-exponent}+m_h}\cO_{x,p}/p^{Nh-h}\cO_{x,p}, 
\end{align*}
which is non-trivial if $N>h+\refC{elementindex-exponent}+m_h$.
In particular, if all of the above conditions on $N$ hold then $\Lambda$ contains the quotient 
\begin{align*}
p^{Nh-h-1}\cO_{x,p}/p^{Nh-h}\cO_{x,p}
\end{align*}
and so, by Nakayama's lemma, the span of the $y_j$ in $\cO_{x,p}$ contain $p^{Nh-h-1}\cO_{x,p}$. Since the $y_j\in R_{x,p}$, we conclude
\begin{align*}
[\cO_{x,p}:R_{x,p}]\leq p^{Nh^2-h^2-h},
\end{align*}
which finishes the proof. 
\end{proof}

Now we are ready to finish the proofs of \cref{mainindex,main-bound}.

\begin{proof}[Proof of \cref{mainindex}]
We begin by reducing to the case in which $K \subset \gG(\bA_f)$ is a direct product of compact open subgroups $K_p \subset \gG(\bQ_p)$.
To see that this is sufficient, first note that an arbitrary compact open subgroup $K \subset \gG(\bA_f)$ always contains such a subgroup $K' = \prod_p K'_p$.
Replacing $K$ by $K'$ changes $i(\gM)$ by at most the (finite) number of primes for which
\[ K \cap \gG(\bQ_p) \neq K' \cap \gG(\bQ_p), \]
and multiplies $[K_\gM^m:K_\gM]$ by $[K_\gM:K'_\gM]$, which is at most $[K:K']$.
Hence replacing $K$ by $K'$ changes the right hand side of the bound in \cref{mainindex} by a uniformly bounded quantity, so it suffices to prove the theorem with $K'$ in place of $K$.

We may therefore assume that $K = \prod_p K_p$, so
\[ [K_\gM^m:K_\gM] = \prod_p [K^m_{\gM,p}:K_{\gM,p}]. \]

For any integer $x$ and any prime $p$, let $[x]_p$ denote the largest power of $p$ which divides $x$ i.e.
\[ [x]_p = p^{v_p(x)}. \]

We will show that we can choose constants $\refC{index-torus-exponent}$ and $\refC{index-disc-exponent}$ uniformly and a constant $\newC{index-small-bound}$ depending on $\gG$, $X$, $\rho$, $K$, $\refC{index-badprime-base}$ such that for every prime $p$, either
\[ p \text{ and } [\cO_x:R_x]_p \text{ are both less than } \refC{index-small-bound} \]
or
\[ [\cO_x:R_x]_p \leq \refC{index-badprime-base}^{i(\gM,p)} [K^m_{\gM,p}:K_{\gM,p}]^{\refC{index-torus-exponent}} [\disc L]_p^{\refC{index-disc-exponent}} \tag{*} \label{eqn:per-prime-inequality} \]
where $i(\gM,p) = 1$ if $K^m_{\gM,p} \neq K_{\gM,p}$ and $0$ if $K^m_{\gM,p} = K_{\gM,p}$.
The bound in \cref{mainindex} follows by taking the product over all primes $p$ (the first case leads to the constant $\refC{index-multiplier}$).

\begin{itemize}
\item If $p$ divides $[K^m_{\gM,p}:K_{\gM,p}]$ then \cref{indexexp} implies that
\[ [\cO_x:R_x]_p \leq [K^m_{\gM,p}:K_{\gM,p}]^{2\refC{indexexp}}. \]

\begin{itemize}
\item If we suppose further that $p^{\refC{indexexp}} \geq \refC{index-badprime-base}^{-1}$, then
\[ [\cO_x:R_x]_p \leq \refC{index-badprime-base} [K^m_{\gM,p}:K_{\gM,p}]^{3\refC{indexexp}} \]
and so taking  $\refC{index-torus-exponent} \geq 3\refC{indexexp}$ suffices.
\end{itemize}

\item If $p$ does not divide $[K^m_{\gM,p}:K_{\gM,p}]$ then by \cref{indexexp},
\[ [\cO_x:R_x]_p \leq p^{\refC{indexexp}}. \]

\begin{itemize}
\item If $[K^m_{\gM,p}:K_{\gM,p}] > 1$, then by~\cite{uy:andre-oort} Proposition 3.15 there exists a uniform constant $\newC{torus-index-multiplier} \leq 1$ such that
\[ \refC{torus-index-multiplier}p \leq [K^m_{\gM,p}:K_{\gM,p}]. \]
If also $p$ is large enough that $p^{\refC{indexexp}} \geq \refC{torus-index-multiplier}^{-2\refC{indexexp}} \refC{index-badprime-base}^{-1}$, we deduce that
\[ [\cO_x:R_x]_p \leq \refC{index-badprime-base} [K^m_{\gM,p}:K_{\gM,p}]^{2\refC{indexexp}}. \]
Thus in this case, $\refC{index-torus-exponent} \geq 2\refC{indexexp}$ suffices.

\item If $[K^m_{\gM,p}:K_{\gM,p}] = 1$ and $p$ is ramified in $L$ then $p$ divides $\disc L$ and $i(\gM,p) = 0$.
Thus
\[ [\cO_x:R_x]_p \leq \refC{index-badprime-base}^{i(\gM,p)} [\disc L]_p^{\refC{indexexp}} \]
and so $\refC{index-disc-exponent} \geq \refC{indexexp}$ suffices.

\item If $[K^m_{\gM,p}:K_{\gM,p}] = 1$, $p$ is unramified in $L$ and $p\notin\Sigma$, then \cref{useUY} implies that
$[\cO_x:R_x] = 1$.
Again $i(\gM, p) = 0$ so both sides of \eqref{eqn:per-prime-inequality} are equal to \( 1 \).
\end{itemize}

\item We are left with the cases
\[ p\in\Sigma \text{ or } p^{\refC{indexexp}} < \refC{torus-index-multiplier}^{-2\refC{indexexp}} \refC{index-badprime-base}^{-1} \]
(note that this includes $p^{\refC{indexexp}} < \refC{index-badprime-base}^{-1}$ because $\refC{torus-index-multiplier} \leq 1$).
In these cases we have an upper bound for $p$ depending on uniform data and on \( \refC{index-badprime-base} \) (recall that \( \Sigma \) is finite and uniform).
The fact that $[\cO_x:R_x]_p \leq p^{\refC{indexexp}}$ implies a similar bound for $[\cO_x:R_x]_p$.
\qedhere
\end{itemize}
\end{proof}

\begin{proof}[Proof of \cref{main-bound}]
By Theorem \ref{height-disc-r-bound}, there are constants $\refC{height-disc-r-multiplier}$ and $\refC{height-disc-r-exponent}$ depending only on \( \gG \), \( X \), \( \cF \) and \( \rho \) such that
\[ H(x) \leq \refC{height-disc-r-multiplier} \abs{\disc R_x}^{\refC{height-disc-r-exponent}}. \]
and
\[ \disc R_x=[\cO_x:R_x]^2\disc\cO_x=[\cO_x:R_x]^2\prod_{i=1}^s\disc F_i. \]
Since each $F_i$ is isomorphic to a subfield of $L$, we know that $\abs{\disc F_i}$ is at most $\abs{\disc L}$. Clearly, the number of isotypic components $s$ is at most the dimension of~$E$. Therefore, \cref{main-bound} follows from \cref{mainindex}.
\end{proof}

\section*{Acknowledgements}
Both authors would like to thank Emmanuel Ullmo and Andrei Yafaev for suggesting that they work together on this problem and for numerous conversations regarding the subject of this paper. The authors also owe a special thank you to Philipp Habegger who suggested the use of Chow polynomials to prove Proposition~\ref{singular-bound}.
They are grateful to Ziyang Gao for pointing out the issue with quantifiers in \cref{main-bound} which is needed for it to imply \cref{andre-oort-conditional}.
Both authors would like to thank the referee for their reading of the manuscript and their helpful comments.

This paper was published in \textit{Mathematische Annalen}, 2016, volume~365, pages 1305--1357.
The published version is available at Springer via \url{http://dx.doi.org/10.1007/s00208-015-1328-3}.
The \texttt{arXiv} version contains a minor correction to the proof of \cref{root-bound}, by means of the additional Lemma~\ref{correction:um-roots}.
This proof is incorrect in the published version.

The first author is indebted to the Engineering and Physical Sciences Research Council and the Institut des Hautes Études Scientifiques for their financial support.
The second author was funded by European Research Council grant 307364 ``Some problems in Geometry of Shimura varieties.''

\bibliographystyle{alpha}
\bibliography{heights5}

\end{document}